\documentclass[a4paper,11pt]{article}
\usepackage{graphicx}
\usepackage{amssymb}
\usepackage[french, english]{babel}
\usepackage[utf8]{inputenc}
\usepackage{amsmath,amsfonts,amssymb,amsthm}
\usepackage[T1]{fontenc}
\usepackage{fullpage}
\usepackage{hyperref}
\usepackage{mathrsfs}
\usepackage{relsize}
\usepackage{tikz}
\usepackage{bbm}
\usepackage{appendix}
\usepackage{pifont}
\usetikzlibrary{patterns}

\definecolor{darkgreen}{rgb}{0.0, 0.5, 0.0}

\newcommand{\eps}{\varepsilon}

\newcommand{\loc}{\mathrm{loc}}

\newcommand{\T}{\mathcal{T}}

\newcommand{\TT}{\mathbb{T}}
\newcommand{\N}{\mathbb N}

\newcommand{\R}{\mathbb R}

\newcommand{\E}{\mathbb E}

\renewcommand{\P}{\mathbb P}

\newcommand{\occ}{\mathrm{occ}}

\theoremstyle{definition}

\newtheorem{thm}{Theorem}

\newtheorem{defn}{Definition}%[section]

\newtheorem{prop}[defn]{Proposition}
\newtheorem{corr}[defn]{Corollary}

\newtheorem{lem}[defn]{Lemma}
\newtheorem{conj}[defn]{Conjecture}

\tikzstyle{every node}=[circle, draw, fill=black!50, inner sep=0pt, minimum width=4pt]
\tikzstyle{white}=[circle, draw, fill=black!0, inner sep=0pt, minimum width=4pt]
\tikzstyle{bigwhite}=[circle, draw, fill=black!0, inner sep=0pt, minimum width=10pt]
\tikzstyle{dual}=[circle, draw=blue, fill=black!0, inner sep=0pt, minimum width=4pt]
\tikzstyle{fat}=[circle, draw, fill=red!50, inner sep=0pt, minimum width=8pt]
\tikzstyle{fat_bis}=[circle, draw, fill=blue!50, inner sep=0pt, minimum width=8pt]
\tikzstyle{fat_ter}=[circle, draw, fill=green!50, inner sep=0pt, minimum width=8pt]
\tikzstyle{rouge}=[circle, draw, fill=red, inner sep=0pt, minimum width=7pt]
\tikzstyle{bleu}=[circle, draw, fill=blue, inner sep=0pt, minimum width=7pt]
\tikzstyle{petitrouge}=[circle, draw, fill=red, inner sep=0pt, minimum width=4pt]
\tikzstyle{petitbleu}=[circle, draw, fill=blue, inner sep=0pt, minimum width=4pt]
\tikzstyle{texte}=[draw=none, fill=none]

\title{\bf{Multi-ended Markovian triangulations and robust convergence to the UIPT}}
\author{Thomas \bsc{Budzinski}}

\begin{document}

\maketitle

\begin{abstract}
We classify completely the infinite, planar triangulations satisfying a weak spatial Markov property, without assuming one-endedness nor finiteness of vertex degrees. In particular, the Uniform Infinite Planar Triangulation (UIPT) is the only such triangulation with average degree $6$.
As a consequence, we prove that the convergence of uniform triangulations of the sphere to the UIPT is robust, in the sense that it is preserved under various perturbations of the uniform measure. As another application, we obtain large deviation estimates for the number of occurencies of a pattern in uniform triangulations.
\end{abstract}

\section*{Introduction}

\paragraph{Local limits of random triangulations.}
Local limits of random planar maps have been the object of extensive study in the last fifteen years. The starting point of this theory was the convergence of large uniform triangulations to the Uniform Infinite Planar Triangulation (UIPT) proved by Angel and Schramm~\cite{AS03}. Since then, similar results have been proved for many other classes of maps such as quadrangulations~\cite{CD06}, or maps with Boltzmann weights on the face degrees (\cite{BS13} in the bipartite case and \cite{St18} in the general case). See also~\cite{C-StFlour} for a complete survey. All these results rely heavily on a very good understanding of the combinatorics of finite maps. More precisely, the proofs use either exact enumeration results going back to Tutte such as~\cite{Tut62}, or bijections with simpler objects such as labelled trees \cite{CV81}.

\paragraph{Spatial Markov property.}
The UIPT exhibits fractal-like properties (e.g. volume growth in $r^4$) whose study was started by Angel in~\cite{Ang03}, relying on an exploration procedure called \emph{peeling}. The key feature of the UIPT allowing to perform such explorations is its \emph{spatial Markov property}: the probability to observe a finite triangulation $t$ as a neighbourhood of the root (we give precise definitions in Section~\ref{subsec_defns} below) in the UIPT only depends on the perimeter and volume of $t$, and not on its geometry. This guarantees that during an exploration of the UIPT, the perimeter and volume of the explored region follow a Markov chain, which reduces the study of the UIPT to the analysis of an explicit Markov chain with values in $\N^2$. Moreover, the spatial Markov property is easy to observe on many natural finite models (such as uniform triangulations of the sphere with a fixed size) and passes well to local limits. We can therefore expect the limits of many natural finite map models to exhibit the spatial Markov property.

\paragraph{Planar Stochastic Hyperbolic Triangulations.}
Partly motivated by this remark, Curien introduced\footnote{The PSHT defined in~\cite{CurPSHIT} are the type-II PSHT, i.e. where multiple edges are allowed but self-loops are forbidden. The adaptation to the type-I case (with both multiple edges and self-loops) was done in~\cite{B16}.} in~\cite{CurPSHIT} a family of random infinite triangulations of the plane called Planar Stochastic Hyperbolic Triangulations (PSHT) (see also~\cite{AR13} for similar objects in a half-plane setting). The PSHT form a one-parameter\footnote{In~\cite{CurPSHIT,B16}, the construction is only given for $\lambda \in (0,\lambda_c]$. The extension to the case $\lambda=0$, where vertices have infinite degrees, will be done carefully in the present work, see Section~\ref{sec_prelim}.} family $\left( \TT_{\lambda} \right)_{0 \leq \lambda \leq \lambda_c}$, where $\lambda_c=\frac{1}{12\sqrt{3}}$ and $\TT_{\lambda_c}$ is just the UIPT. They are characterized by the following equation: if $t$ is a triangulation with one hole of perimeter $p$ and inner volume $v$, then
\[ \P \left( t \subset \TT_{\lambda} \right) = C_p(\lambda) \times \lambda^v, \]
where $t \subset T$ means that $t$ is a neighbourhood of the root in $T$, and the constants $C_p(\lambda)$ are explicit. In particular, the $\TT_{\lambda}$ satisfy the spatial Markov property. For $\lambda<\lambda_c$, the PSHT $\TT_{\lambda}$ exhibit a hyperbolic behaviour contrasting sharply with the UIPT (e.g. exponential volume growth, positive speed of the simple random walk)~\cite{CurPSHIT}. The PSHT were proved in~\cite{BL19} to be the local limits of uniform triangulations with genus proportional to their sizes.

\paragraph{Characterizing Markovian triangulations.}
An important step of the argument of~\cite{BL19} was the following characterization.

\begin{prop}\cite[Theorem 2]{BL19}\label{prop_oneended_case}
Let $T$ be a random, infinite, one-ended planar triangulation with finite vertex degrees. If $T$ satisfies the spatial Markov property, then $T$ is a mixture of PSHT, i.e. $T$ is of the form $\TT_{\Lambda}$, where $\Lambda$ is a random variable with values in $(0,\lambda_c]$.
\end{prop}
 
Indeed, once tightness and planarity and one-endedness of the limits are established, this result ensures that any subsequential limit is a mixture of PSHT. However, although~\cite{BL19} does not rely on precise combinatorial asymptotics on the finite models, checking the assumptions of one-endedness and planarity on subsequential limits still required a strong combinatorial input, namely the Goulden--Jackson recursion of~\cite{GJ08}. The goal of the present work is to remove the assumptions of one-endedness and finite vertex degrees in Proposition~\ref{prop_oneended_case}, and therefore to make the strategy of~\cite{BL19} less reliant on the combinatorial understanding of the models. However, the results presented here do not allow to shorten significantly the proof of~\cite{BL19} (this is discussed in Section~\ref{sec_conj}).

\paragraph{Triangulations with infinite vertex degrees.}
In this work, an infinite triangulation will be a connected, planar gluing of a countable collection of triangles along their vertices and edges (see Section~\ref{subsec_defns} for a more formal definition). We allow loops and multiple edges and, importantly, we do not require the vertex degrees to be finite, which is quite unusual in the literature. The point of this extended definition will be to allow us to "skip" the tightness step in the proofs of local convergence results. In particular, given two corners $c_1$ and $c_2$ incident to the same vertex $v$, it may be necessary to cross infinitely many edges to go from $c_1$ to $c_2$ in a small neighbourhood of $v$ (see the right part of Figure~\ref{fig_two_degenerate_examples}).

In the context of multi-ended triangulations, we need to extend the definition of the spatial Markov property. We say that a random infinite, rooted, planar triangulation $T$ is \emph{Markovian} if for any finite triangulation $t$ with holes, the probability $\P \left( t \subset T \right)$ only depends on the perimeters of the holes of $t$ and on its total number of faces.

We finally introduce two important examples of "degenerate" planar triangulations $\TT_0$ and $\TT_{\star}$ with infinite vertex degrees. We denote by $\TT_0$ the triangulation obtained by gluing the edges of the triangles along the structure of a complete binary tree (as on the left of Figure~\ref{fig_two_degenerate_examples}). This is also the natural way to extend the PSHT $\TT_{\lambda}$ to the case $\lambda=0$. On the other hand $\TT_{\star}$ (on the right of Figure~\ref{fig_two_degenerate_examples}) is the only infinite, planar triangulation with only one vertex: it is obtained by forming a triangle with three loops, and then recursively adding two loops inside of each loop to form new triangles. We highlight that our definition considers $\TT_0$ and $\TT_{\star}$ as two \emph{distinct} objects, even if they have the same dual. Indeed, the vertices are not glued in the same way in $\TT_0$ and in $\TT_{\star}$. We are now able to state our main theorem.

\begin{figure}
\begin{center}
\begin{tikzpicture}
\draw(90:1) -- (210:1)--(330:1)--(90:1);
\draw[red, thick, ->] (210:1)--(0,-0.5);
\draw[red, thick] (210:1)--(330:1);
\draw(90:1)--(150:1.5)--(210:1)--(270:1.5)--(330:1)--(30:1.5)--(90:1);
\draw(30:1.5)--(60:1.6)--(90:1)--(120:1.6)--(150:1.5)--(180:1.6)--(210:1)--(240:1.6)--(270:1.5)--(300:1.6)--(330:1)--(0:1.6)--(30:1.5)--(60:1.6)--(90:1);
\draw(90:1)--(100:1.6)--(120:1.6)--(135:1.9)--(150:1.5)--(165:1.9)--(180:1.6)--(200:1.6)--(210:1);
\draw(210:1)--(220:1.6)--(240:1.6)--(255:1.9)--(270:1.5)--(285:1.9)--(300:1.6)--(320:1.6)--(330:1);
\draw(330:1)--(340:1.6)--(0:1.6)--(15:1.9)--(30:1.5)--(45:1.9)--(60:1.6)--(80:1.6)--(90:1);

\draw(90:1)node{};
\draw(210:1)node{};
\draw(330:1)node{};
\draw(150:1.5)node{};
\draw(270:1.5)node{};
\draw(30:1.5)node{};

\draw(60:1.6)node{};
\draw(120:1.6)node{};
\draw(180:1.6)node{};
\draw(240:1.6)node{};
\draw(300:1.6)node{};
\draw(0:1.6)node{};

\draw(80:1.6)node{};
\draw(100:1.6)node{};
\draw(200:1.6)node{};
\draw(220:1.6)node{};
\draw(320:1.6)node{};
\draw(340:1.6)node{};

\draw(15:1.9)node{};
\draw(45:1.9)node{};
\draw(135:1.9)node{};
\draw(165:1.9)node{};
\draw(255:1.9)node{};
\draw(285:1.9)node{};

\begin{scope}[shift={(8,-1)}]
\draw(0,0) to[out=180,in=247.5] (157.5:1.4) to[out=67.5,in=135] (0,0);
\draw(0,0) to[out=135,in=202.5] (112.5:1.4) to[out=22.5,in=90] (0,0);
\draw(0,0) to[out=90,in=157.5] (67.5:1.4) to[out=-22.5,in=45] (0,0);
\draw(0,0) to[out=45,in=112.5] (22.5:1.4) to[out=-67.5,in=0] (0,0);

\draw[red, thick](0,0) to[out=180,in=270] (160:1.7) to[out=90,in=180] (110:1.7) to[out=0,in=90] (0,0);
\draw[red, thick,->](0,0) to[out=180,in=270] (160:1.7) to[out=90,in=180] (110:1.7);
\draw(0,0) to[out=90,in=180] (70:1.7) to[out=0,in=90] (20:1.7) to[out=-90,in=0] (0,0);

\draw(0,0) to[out=195,in=270] (160:2) to[out=90,in=180] (90:2) to[out=0,in=90] (20:2) to[out=270,in=-15] (0,0);

\draw(0,0) to[out=-22.5,in=45] (-45:1.4) to[out=225.5,in=-67.5] (0,0);
\draw(0,0) to[out=210,in=270] (160:2.2) to[out=90,in=180] (90:2.2) to[out=0,in=90] (20:2.2) to[out=270,in=0] (-60:1.5) to[out=180,in=-90] (0,0);

\draw(0,0)node{};
\end{scope}
\end{tikzpicture}
\end{center}
\caption{The beginning of the construction of $\TT_0$ (on the left) and $\TT_{\star}$ (on the right). Note that these are two different triangulations: $\TT_0$ has infinitely many vertices whereas $\TT_{\star}$ has only one.}\label{fig_two_degenerate_examples}
\end{figure}
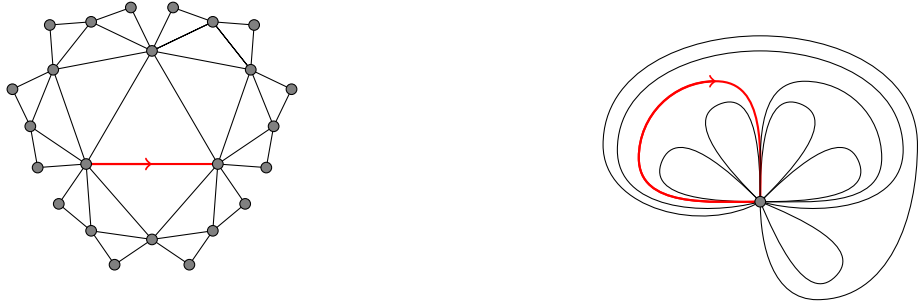

\begin{thm}\label{main_thm}
Let $T$ be a Markovian infinite, planar triangulation. Then $T$ is of the form $\TT_{\Lambda}$, where $\Lambda$ is a random variable with values in $[0,\lambda_c] \cup \{ \star \}$.
\end{thm}

In particular, this result means that there is no natural notion of "uniform" multi-ended planar triangulation. This was already observed by Linxiao Chen for a much stronger version of the Markov property~\cite{ChPerso}.

\paragraph{Applications to the convergence of "perturbed" models to the UIPT.}
As explained above, the main motivation behind this result is to be able to prove the convergence of finite models to the UIPT without needing a good control on the combinatorics of these models. Indeed, a corollary of Theorem~\ref{main_thm} is that the UIPT is the only Markovian infinite planar triangulation with average vertex degree $6$. This characterization of the UIPT involves only properties that are usually easy to observe on finite models, and is therefore useful to prove the convergence of such models. Roughly speaking, the "meta-theorem" is that any "almost-uniform" triangulation model with an additional structure which is "small compared to the size" converges to the UIPT. As an illustration, we present three particular cases: triangulations with defects, with moderate genus, and with a very high temperature Ising model. Note that our approach is robust enough to treat a model where these three perturbations would be combined.

\paragraph{Triangulations with defects.}
We first consider "triangulations with defects", i.e. where a small proportion of the faces are not triangles. More precisely, for $\mathbf{f}=(f_j)_{j \geq 1}$ such that $\sum_{j \geq 1} j f_j$ is an even integer, we denote by $M_{\mathbf{f}}$ a uniform rooted planar map among all those with exactly $f_j$ faces of degree $j$ for all $j \geq 1$. We write $|\mathbf{f}|=\frac{1}{2} \sum_{j \geq 1} j f_j$ for the number of edges of such a map.

\begin{corr}\label{cor_triang_with_defects}
Let $\mathbf{f}^n$ for $n \geq 1$ be face degree sequences such that $|\mathbf{f}^n| \to +\infty$ when $n \to +\infty$. We assume that $\sum_{j \ne 3} j f_j = o \left( |\mathbf{f}^n| \right)$. Then the random maps $M_{\mathbf{f}^n}$ converge locally to the UIPT as $n \to +\infty$.
\end{corr}

We note that the combinatorics of maps with prescribed face degrees are only fully understood when we consider bipartite maps~\cite{MM07}, which is not the case here. On the other hand, the case where the non-triangular faces are simple and face-disjoint was completely treated by Krikun~\cite{Kri07}. This Corollary was one of the motivations for this work, and will be used in a forthcoming work of Curien, Kortchemski and Marzouk on mesoscopic limits of uniform planar maps with fixed numbers of edges, faces and vertices~\cite{CKM21}.

\paragraph{Triangulations with moderate genus.}
For $n \geq 1$ and $0 \leq g \leq \frac{n}{2}$, we denote by $T_{n,g}$ a uniform rooted triangulation with $2n$ faces and genus $g$. Note that the condition $g \leq \frac{n}{2}$ is necessary for such a triangulation to exist. Using Theorem~\ref{main_thm}, we can recover the following particular case of the main result of~\cite{BL19}.

\begin{corr}\label{cor_moderate_genus}
Let $(g_n)$ be a sequence such that $\frac{g_n}{n} \to 0$ as $n \to +\infty$. Then $T_{n,g_n}$ converges locally to the UIPT as $n \to +\infty$.
\end{corr}

Compared to~\cite{BL19}, the proof is now much shorter, and does not rely on any combinatorics of high genus maps.

\paragraph{Very high temperature Ising model.}
Our last application deals with the Ising model on random triangulations. If $t$ is a triangulation of the sphere and $\sigma$ is a colouring of the faces of $t$ in black and white, we denote by $H(t,\sigma)$ the number of edges $e$ of $t$ such that both sides of $e$ have the same colour. For $\beta \in \R$ and $n \geq 1$, we denote by $T_n[\beta]$ the random face-coloured triangulation of the sphere with $2n$ faces such that for all $(t, \sigma)$ the probability that $T_n[\beta]=(t,\sigma)$ is proportional to $e^{\beta H(t,\sigma)}$.

\begin{corr}\label{cor_Ising}
Let $(\beta_n)$ be a sequence of real numbers such that $\beta_n \to 0$ as $n \to +\infty$. Then $T_n[\beta_n]$ converges locally as $n \to +\infty$ to the UIPT equipped with Bernoulli face percolation with parameter $\frac{1}{2}$.
\end{corr}

The exact same argument would also work for the Potts model or similar statistical physics models. We note that the case where the Ising model lives on the vertices and the inverse temperature $\beta>0$ is fixed was treated in~\cite{AMS20}. We also refer to~\cite{CT20, CT20b, CT20c} for local convergence results on Ising triangulations with a boundary, once again for $\beta$ fixed.

\paragraph{Large deviations for pattern occurences in uniform triangulations.}
We conclude the consequences of Theorem~\ref{main_thm} with a last application to the counting of patterns in uniform triangulations. Let $t_0$ be a fixed rooted, finite triangulation with one or several holes. For every triangulation $t$ of the sphere, we denote by $\occ_{t_0}(t)$ the number of occurences of $t_0$ in $t$. More precisely $\occ_{t_0}(t)$ is the number of oriented edges $e$ of $t$ such that there is a neighbourhood of $e$ isomorphic to $t_0$, so that $e$ is matched with the root of $t_0$. We recall that $T_n$ stands for a uniform triangulation of the sphere with $2n$ faces, and $\TT_{\lambda_c}$ for the UIPT.

\begin{thm}\label{thm_large_deviations}
	For every finite triangulation with holes $t_0$ and every $\eps>0$, the probability
	\[ \P \left( \left| \frac{\occ_{t_0}(T_n)}{6n} - \P \left( t_0 \subset \TT_{\lambda_c} \right) \right|> \eps  \right)\]
	decays exponentially in $n$.
\end{thm}

Note that the local convergence of $T_n$ to $\TT_{\lambda_c}$ only gives the convergence of the expectation of $\frac{\occ_{t_0}(T_n)}{6n}$ to $\P \left( t_0 \subset \TT_{\lambda_c} \right)$. Convergence in probability was proved in~\cite{BD18} in the more general context of planar maps with Boltzmann weights on the face degrees. The proof of Theorem~\ref{thm_large_deviations} consists of showing using Theorem~\ref{main_thm} that if $\beta_n \to 0$, then $T_n$ biased by $\exp (\beta_n \, \occ_{t_0}(T_n))$ still converges to the UIPT. This means that perturbating the uniform measure by subexponential factors does not affect $\occ_{t_0}(T_n)$ in a significant way, so triangulations where $\occ_{t_0}(T_n)$ deviates from its mean must be exponentially rare.

\paragraph{Sketch of proof of the main theorem.}
The starting point of the proof of Theorem~\ref{main_thm} is roughly the same as the proof of~\cite[Theorem 2]{BL19}: using linear equations between the probabilities $\P \left( t \subset T \right)$ and the Hausdorff moment problem, we express the probabilities $\P \left( t \subset T \right)$ as the moments of a pair $(\Lambda, \Gamma)$ of random variables. Here $\Lambda$ and $\Gamma$ can roughly be interpretated as Boltzmann weights on respectively the number of vertices and the number of ends. The PSHT $\TT_{\lambda}$ for $\lambda \in [0,\lambda_c]$ corresponds to the case $(\Lambda, \Gamma)=(\lambda,0)$, whereas the degenerate triangulation $\TT_{\star}$ corresponds to $(\Lambda, \Gamma)=(0,1)$. We then use the explicit generating function of triangulations with a boundary to prove that one of the probabilities $\P \left( t \subset T \right)$ is negative, unless almost surely $\Gamma=0$ or $(\Lambda, \Gamma)=(0,1)$. Finally, we refer to Section~\ref{sec_conj} for a discussion on extensions of Theorem~\ref{main_thm} and a conjecture in the nonplanar case.

\paragraph{Acknowledgements.}
The author thanks Linxiao Chen, Nicolas Curien, Igor Kortchemski and Cyril Marzouk for useful discussions. We are grateful to the Laboratoire de Mathématiques d'Orsay, where this work was started, for its hospitality.

\tableofcontents

\section{Preliminaries}\label{sec_prelim}

\subsection{Basic definitions}\label{subsec_defns}

\paragraph{Finite triangulations.}
We recall that a \emph{finite planar map} is a gluing of a finite number of finite polygons which is homeomorphic to the sphere. We will always consider \emph{rooted} maps, which means that they carry a distinguished oriented edge called the \emph{root}. We will call \emph{root vertex} the starting point of the root edge, and \emph{root face} the face lying on the left of the root edge. A \emph{triangulation of the sphere} is a finite planar map where all the faces have degree $3$. We denote by $\T_f$ the set of triangulations of the sphere. Note that in all this work, we will consider \emph{type-I triangulations}, which means that we allow self-loops and multiple edges.

A \emph{triangulation with holes} is a finite, planar map $t$ with marked faces called the \emph{holes}, such that:
\begin{itemize}
	\item all the internal faces (i.e. the faces which are not holes) are triangles;
	\item every edge is incident to at least one internal face, and the set of internal faces forms a connected subset of the dual of $t$.
\end{itemize}
It will also be important for us to have a notion of volume for triangulations with holes. While the most usual convention is to use the total number of vertices, we will choose a different one here, which allows the volume to take any nonnegative value\footnote{The point of this change is to have a definition of the PSHT which still makes sense for $\lambda=0$, see Section~\ref{subsec_combinatorics}.}.

\begin{lem}
Let $t$ be a triangulation with $k$ holes of perimeters $p_1, \dots, p_k$, and let $\widetilde{v}$ be its total number of vertices. Then we have
\begin{equation}\label{eqn_defn_inner_volume}
v := \widetilde{v}-1-\sum_{i=1}^k (p_i-1) \geq 0.
\end{equation}
We call this quantity $v$ the \emph{inner volume} of $t$.
\end{lem}

\begin{proof}
We denote by $V_{\partial}$ the set of vertices of $t$ which lie on at least one of the holes. Let $\widehat{t}$ be the map whose vertices are the holes of $t$ and the vertices of $V_{\partial}$, and where for each hole $h$ of $t$ and each vertex $v$ on $\partial h$, there is one edge linking $h$ to $v$. Then a cycle in $\widehat{t}$ would disconnect the set of internal faces, so $\widehat{t}$ is a forest. On the other hand $\widehat{t}$ has $k+\# V_{\partial}$ vertices and $\sum_{i=1}^k p_i$ edges, so
\[ k+ \# V_{\partial} \geq 1+\sum_{i=1}^k p_i, \]
and finally $\widetilde{v} \geq \# V_{\partial} \geq 1+\sum_{i=1}^k (p_i-1)$.
\end{proof}

On the other hand, let $p_1, \dots, p_k \geq 1$ and $v \geq 0$. We note that there is a triangulation with $k$ holes of perimeters $p_1, \dots, p_k$ and inner volume $v$ except if $v=0$ and either $k=1$ and $p_1 \in \{1,2\}$, or $k=2$ and $p_1=p_2=1$. To bridge this gap, we set the following conventions:
\begin{itemize}
	\item the only triangulation with one hole of perimeter $1$ and inner volume $0$ is the triangulation $t^1_0$ consisting of a single vertex and no edge;
	\item the only triangulation with one hole of perimeter $2$ and inner volume $0$ is the triangulation $t^2_0$ consisting of two vertices linked by a single edge;
	\item the only triangulation with two holes of perimeter $1$ and inner volume $0$ is the triangulation $t^{1,1}_0$ consisting of a single vertex and a loop on this vertex.
\end{itemize}

For $p \geq 1$, a \emph{finite triangulation of the $p$-gon} is a finite, planar map where the root face has degree $p$ and is simple, and all the other faces have degree $3$. We denote by $\tau_n(p)$ the number of triangulations of the $p$-gon with volume $n$, i.e. with $n+1$ vertices overall. This convention for the volume is designed so that when we use a triangulation of a polygon to fill one of the holes of a triangulation with holes, their volumes add up.

\paragraph{Sub-triangulations and balls.}
If $T$ is a triangulation of the sphere and $t$ a triangulation with holes, we write $t \subset T$ if $T$ can be obtained from $t$ by filling each hole of $t$ with a triangulation of a polygon\footnote{By convention, we always have $t^1_0 \subset T$, we have $t^2_0 \subset T$ if and only if the root of $T$ is not a loop, and $t^{1,1}_0 \subset T$ if and only if the root of $T$ is a loop.}. We call $t$ a \emph{sub-triangulation} of $T$ write $t \subset T$. We now define two families of sub-triangulations of particular interest.

If $T$ is a triangulation of the sphere and $r \geq 0$, we denote by $B_r(T)$ the sub-triangulation of $T$ consisting of all the faces of $T$ incident to at least one vertex at graph distance at most $r-1$ from the root vertex, together with all the vertices and edges incident to these faces. We also denote by $B_r^*(T)$ the sub-triangulation consisting of all the faces at distance at most $r$ from the root face in the dual graph of $T$, together with all the vertices and edges incident to these faces. We call $B_r(T)$ (resp. $B_r^*(T)$) the \emph{ball of radius $r$} (resp. \emph{dual ball of radius $r$}) of $T$. In some cases, we will also use the notation $B_r(T;e)$ or $B_r^*(T;e)$ for the ball or dual ball around the root edge $e$ to emphasize the choice of the root edge.

\paragraph{Local distances and infinite triangulations.}
We now define two versions of the local topology on the set of triangulations. The first one is the one which is used most of the time in the literature, while the second is a weaker version. We will mostly use the second one in this work, but we will still obtain convergence results for the first one in the end (Corollaries~\ref{cor_triang_with_defects}, \ref{cor_moderate_genus} and \ref{cor_Ising}). If $T,T'$ are two triangulations of the sphere, we write
\[ d_{\loc}(T,T') = \left( 1+\min \{ r \geq 0 | B_r(T) \ne B_r(T') \}\right)^{-1},\]
\[ d_{\loc}^*(T,T') = \left( 1+\min \{ r \geq 0 | B_r^*(T) \ne B_r^*(T') \}\right)^{-1}.\]
We call $d_{\loc}$ (resp. $d^*_{\loc}$) the \emph{local distance} (resp. \emph{dual local distance}) on the set of triangulations. We denote by $\overline{\T}$ the completion of $\T_f$ for $d^*_{\loc}$, and write $\T_{\infty} = \overline{\T} \backslash \T_f$. An element of $\T_{\infty}$ will be called an \emph{infinite triangulation}. Alternatively, an infinite triangulation is a planar, connected gluing of countably many triangles along some of their vertices and edges, such that all the edges are glued two by two. Note however that vertex degrees may be infinite. It is also possible that the neighbourhood of a vertex $v$ becomes disconnected if $v$ is removed (see the right of Figure~\ref{fig_two_degenerate_examples}). However, note that the notion of a dual ball still makes sense in a triangulation $T \in \T_{\infty}$. As an example, Figure~\ref{fig_two_degenerate_examples} represents $B^*_3(\TT_0)$ on the left and $B_1^*(\TT_{\star})$ on the right. The notation $t \subset T$ also makes sense in this context (we say that $t \subset T$ if there is $r$ such that $t \subset B_r^*(T)$). Finally, we note that for any $r \geq 0$, a dual ball $B_r(T^*)$ has at most $1+3+\dots+3^r$ faces, so it may only take finitely many values. This implies that $\overline{\T}$ is compact for $d_{\loc}^*$. This will be important in the proofs of Corollaries~\ref{cor_triang_with_defects}, \ref{cor_moderate_genus} and \ref{cor_Ising}.

Finally, also with Corollaries~\ref{cor_triang_with_defects}, \ref{cor_moderate_genus} and \ref{cor_Ising} in sight, we state an easy lemma (see e.g.~\cite[Lemma 3]{BL19}) which will bridge the gap between convergence for $d^*_{\loc}$ and for $d_{\loc}$.

\begin{lem}\label{lem_dual_convergence}
	Let $(T_n)$ be a sequence of triangulations of $\overline{\T}$. Assume that
	\[ T_n \xrightarrow[n \to +\infty]{d_{\loc}^*} T,\]
	where $T \in \overline{\T}$ has only vertices with finite degrees. Then we also have $T_n \to T$ for $d_{\loc}$ when $n \to +\infty$.
\end{lem}

\subsection{The spatial Markov property}

We now define precisely our spatial Markov property for infinite triangulations. We will introduce two definitions of this property and prove that they are equivalent. The difference between the two definitions is that the second includes the knowledge of which hole is filled with an infinite component, while the first one does not. The first definition (Definition~\ref{defn_markov}) is the one that is easy to observe on finite models and will be used to prove Corollaries~\ref{cor_triang_with_defects}, \ref{cor_moderate_genus} and \ref{cor_Ising}. However, in most of this work, we will use the second one, introduced in Lemma~\ref{lem_equivalence_weak_Markov}, which is more convenient to deal with infinite models.

Let $T$ be an infinite, planar triangulation, and let $t$ be a finite triangulation with $k$ holes. We recall that we write $t \subset T$ if there is a neighbourhood of the root in $T$ which is isomorphic to $t$, or equivalently if $T$ can be obtained by filling each hole of $t$ with a finite or infinite triangulation. It follows from the connectedness properties in the definition of a triangulation with holes that if $t \subset T$ then this inclusion is unique, i.e. the neighbourhood of $T$ isomorphic to $t$ is determined by $t$ and $T$.  We also write $t \subset_{\infty} T$ if $t \subset T$ and furthermore each of the holes of $t$ contains infinitely many triangles of $T$.

\begin{defn}\label{defn_markov}
Let $T$ be a random infinite planar triangulation. We say that $T$ is \emph{Markovian} if there are numbers $b^{p_1, \dots, p_k}_v$ for $k \geq 1$, $p_1, \dots, p_k \geq 1$ and $v \geq 0$ such that, for any triangulation $t$ with $k$ holes of perimeters $p_1, \dots, p_k$ and inner volume $v$:
\[ \P \left( t \subset T \right) = b^{p_1, \dots, p_k}_v. \]
\end{defn}

\begin{lem}\label{lem_equivalence_weak_Markov}
A random infinite planar triangulation $T$ is Markovian if and only if the following condition is satisfied. There are numbers $a^{p_1, \dots, p_k}_v$ for $k \geq 1$, $p_1, \dots, p_k \geq 1$ and $v \geq 0$ such that, for any triangulation $t$ with $k$ holes of perimeters $p_1, \dots, p_k$ and inner volume $v$:
\[ \P \left( t \subset_{\infty} T \right) = a^{p_1, \dots, p_k}_v. \]
\end{lem}

\begin{proof}
Assume that $T$ satisfies the condition of the lemma, and let $\left( a^{p_1, \dots, p_k}_v \right)$ be the associated constants. Let $t$ be a triangulation with $k$ holes of perimeters $p_1, \dots, p_k$ and inner volume $v$. Then the probability that $t \subset T$ can be expressed as a sum over all ways to fill some of the holes of $t$ (but not all of them) with finite triangulations of the adequate polygons. We obtain
\[ \P \left( t \subset T \right) = \sum_{\substack{ I \subset \{1, \dots, k\} \\ I \ne \emptyset}} \sum_{v_i \geq 0 \mbox{ for } i \notin I} \left( \prod_{i \notin I} \tau_{v_i}(p_i) \right) a_{v + \sum_{i \notin I} v_i}^{\left( p_i \right)_{i \in I}},\]
where by $i \notin I$ we mean $i \in \{1,\dots,k\} \backslash I$, and we recall that $\tau_v(p)$ is the number of triangulations of the $p$-gon with volume $v$. This only depends on $v$ and the $p_i$, so $T$ is Markovian.

Now let $T$ be a Markovian triangulation. For any $\ell \geq 0$ and any triangulation $t$ with $k \geq \ell$ holes, we denote by $t \subset_{\infty}^{\ell} T$ the event that $t \subset T$ and for every $i \leq \ell$, the $i$-th hole of $t$ contains infinitely many faces of $T$. We will prove by induction on $\ell$ that $\P \left( t \subset_{\infty}^{\ell} T \right)$ only depends on $\ell$ and the perimeters and inner volume of $t$. The initialization for $\ell=0$ holds because $T$ is Markovian, whereas the case $\ell=k$ will prove the lemma. For the induction step, assume the result holds for some $\ell \geq 0$, and let $t$ be a triangulation with $k \geq \ell+1$ holes of perimeters $p_1, \dots, p_k$. Then the induction follows from the identity
\[ \P \left( t \subset_{\infty}^{\ell+1} T \right) = \P \left( t \subset_{\infty}^{\ell} T \right) - \sum_{t'} \P \left( t \oplus_{\ell+1} t' \subset_{\infty}^{\ell} T \right), \]
where the sum is over all finite triangulations $t'$ of the $p_{\ell+1}$-gon, and by $t \oplus_{\ell+1} t'$ we mean the triangulation with $k-1$ holes obtained by filling the $(\ell+1)$-th hole of $t$ with $t'$. The set of values of $t'$ only depends on $p_{\ell+1}$ and each term on the right-hand side only depends on $\ell$ and the perimeters and inner volume of $t$. Therefore, so does the left-hand side, which proves the lemma.
\end{proof}

\subsection{Combinatorics and infinite models}\label{subsec_combinatorics}

\paragraph{Counting triangulations of polygons.}
We now recall the exact enumeration of triangulations of polygons. For $n \geq 0$ and $p \geq 1$, we recall that $\tau_n(p)$ is the number of triangulations of the $p$-gon with volume $n$ (i.e. $n+1$ vertices in total). We also write
\[ Z_p(\lambda)= \sum_{n \geq 0} \tau_n(p) \lambda^n \]
for the generating function of triangulations of the $p$-gon, and finally $\mathcal{Z}_{\lambda}(x) = \sum_{p \geq 1} Z_p(\lambda) x^p$. By exact enumeration results of Krikun\footnote{More precisely, using the Euler formula to express the number of vertices (our notion of volume) in terms of the number of edges (Krikun's notion of volume), we have $\mathcal{Z}_{\lambda}(x)=U_0(\lambda^{1/3}, \lambda^{1/3}x)$, where $U_0$ is the explicit function of~\cite[Section 2.2]{Kri07}.}~\cite{Kri07}, we have $Z_p(\lambda)<+\infty$ if and only if $\lambda \leq \lambda_c := \frac{1}{12\sqrt{3}}$. Moreover, for $\lambda \in [0,\lambda_c]$, we have
\begin{equation}\label{eqn_computation_Z}
\mathcal{Z}_{\lambda}(x) = \frac{1}{2} \left( -1+x+\left(1-\frac{1}{\sqrt{1+8h}}x \right) \sqrt{1-\frac{4h}{\sqrt{1+8h}}x} \, \right),
\end{equation}
where $h \in \left[ 0,\frac{1}{4} \right]$ is such that 
\begin{equation}\label{eqn_defn_h}
\lambda=\frac{h}{(1+8h)^{3/2}}.
\end{equation}
Note that the expression of $\mathcal{Z}$ here is slightly different from how it usually appears because of our different volume convention. More precisely,

Finally, note also that $\mathcal{Z}_0(x)=0$.

\paragraph{The UIPT and the PSHT.}
We now define precisely the PSHT $(\TT_{\lambda})_{0 \leq \lambda \leq \lambda_c}$, where we recall that $\lambda_c=\frac{1}{12\sqrt{3}}$. The PSHT were introduced in~\cite{CurPSHIT} in the type-II setting, and their definition was extended to type-I triangulations in~\cite{B16}. They are random one-ended infinite, planar triangulations, characterized by the following property: for any triangulation $t$ with one hole of perimeter $p$ and inner volume $v$, we have
\begin{equation}\label{eqn_defn_PSHT}
\P \left( t \subset \TT_{\lambda} \right) = C_p^{PSHT}(\lambda) \lambda^{v}.
\end{equation}
Moreover, we have $C_1^{PSHT}(\lambda)=1$ and $\left( C_p^{PSHT}(\lambda) \right)_{p \geq 1}$ satisfies the following recursion:
\begin{equation}\label{eqn_Cp_psht}
C_p^{PSHT}(\lambda) = C_{p+1}^{PSHT}(\lambda) + 2 \sum_{i=0}^{p-1} Z_{i+1}(\lambda) C_{p-i}^{PSHT}(\lambda).
\end{equation}
This recursion was obtained in~\cite{CurPSHIT} and it was observed in~\cite{B16} that it yields the exact formula $C_p^{PSHT}(\lambda)=\frac{1}{(1+8h)^{p/2}} \sum_{q=0}^{p-1} \binom{2q}{q} h^q$, where $h$ is given by~\eqref{eqn_defn_h}. Note that our formula for $C_p^{PSHT}(\lambda)$ differs from the one obtained in~\cite{B16} by a factor $\lambda^p$. This is because $v$ now denotes the inner volume instead of the total number of vertices, and the reason for this change of convention is that now each factor in~\eqref{eqn_defn_PSHT} still makes sense for $\lambda=0$ (with the convention of~\cite{B16}, one factor would go to $+\infty$ and the other to $0$).

In particular, the random map $\TT_{\lambda_c}$ is the UIPT. On the other hand, for $\lambda=0$, we get $C_p^{PSHT}(0)=1$ and $\P \left( t \subset \TT_{\lambda} \right)=\mathbbm{1}_{v=0}$. Therefore, the triangulation $\TT_0$ is the dual of a complete binary tree (i.e. the triangulation depicted on the left of Figure~\ref{fig_two_degenerate_examples}).

\paragraph{The one-vertex triangulation $\TT_{\star}$.}
We finally introduce another deterministic example of degenerate Markovian triangulation. We denote by $\TT_{\star}$ the unique infinite planar triangulation with only one vertex. One way to construct $\TT_{\star}$ is to start from three loops on the same vertex forming a triangle, and then recursively add two loops to form a new triangle inside of each loop.

\begin{lem}\label{lem_Tstar}
The triangulation $\TT_{\star}$ is Markovian with
\[a^{p_1, \dots, p_k}_v =
\begin{cases}
1 & \mbox{if $p_1=\dots=p_k=1$ and $v=0$,}\\
0 & \mbox{else.}
\end{cases}\]
\end{lem}

\begin{proof}
First, if $t \subset \TT_{\star}$, all the edges of $\TT_{\star}$ are loops, so all the holes of $t$ must have perimeter $1$. Moreover $\TT_{\star}$ has only one vertex, so the inner volume of $t$ must be $0$ by~\eqref{eqn_defn_inner_volume}. On the other hand, if $t$ is a triangulation with holes of perimeter $1$ and inner volume $0$, then it follows from~\eqref{eqn_defn_inner_volume} that $t$ has only one vertex, and all the edges are loops. By planarity, each loop of $t$ separates $t$ in two, so $t$ consists of loops on the same vertex whose nesting structure is that of a binary tree. This implies $t \subset \TT_{\star}$.
\end{proof}

\section{Proof of the main theorem}

To prove Theorem~\ref{main_thm}, we will rely on the characterization of Markovian triangulations given by Lemma~\ref{lem_equivalence_weak_Markov}. More precisely, let $T$ be a Markovian triangulation and let $a^{p_1, \dots, p_k}_v=\P \left( t \subset_{\infty} T \right)$ for a triangulation $t$ with $k$ holes of perimeters $p_1, \dots, p_k$ and inner volume $v$. Then these coefficients satisfy the following linear equations, that we will call the \emph{peeling equations}:
\[ a^{p_1,\dots,p_k}_v = a^{p_1+1,p_2,\dots,p_k}_{v} + 2 \sum_{i=0}^{p_1-1} \sum_{j \geq 0} a^{p_1-i,p_2,\dots,p_k}_{v+j} \times \tau_j(i+1) + \sum_{i=0}^{p_1-1} a^{i+1, p_1-i, p_2,\dots,p_k}_{v}.\]
To obtain this equation, we first note that for $v=0$, $k=1$ and $p_1=1$, the equation consists of distinguishing whether the root edge is a loop or not, and whether one side of this loop is filled with a finite region. Note also that $a^1_0=1$. In all other cases, consider a triangulation $t$ with $k$ holes of perimeters $p_1, \dots, p_k$ and inner volume $v$, and assume $t \subset T$. We fix an edge $e$ on the boundary of the first hole of $t$ and explore the triangle $f$ of $T \backslash t$ incident to $e$ (see Figure~\ref{fig_peeling_cases}). By planarity, either the third vertex $x$ of $f$ (other than the two ends of $e$) does not belong to $t$, or $x$ lies on the boundary of the first hole. The first term in the peeling equation corresponds to the case where $x$ is not in $t$. The first sum corresponds to the case where $f$ separates $T \backslash t$ into a finite part with perimeter $i+1$ and $j+1$ vertices in total, and an infinite part. The factor $2$ comes from the possibility that the finite component lies either on the left or on the right of $f$. Finally, the second sum corresponds to the case where $f$ splits the first hole into two holes of perimeters $i+1$ and $p_1-i$, each of which contains infinitely many triangles of $T$. This last case is the one which is new compared to the one-ended case.

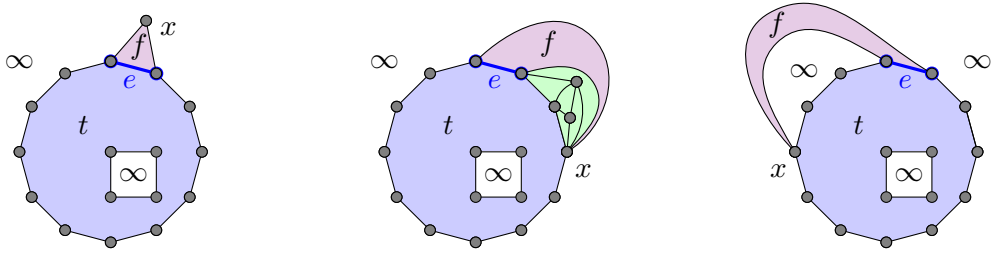
\begin{figure}
	\begin{center}
		\begin{tikzpicture}[scale=1.2]
	\fill[color=violet!20]
		(90:1)--(60:1)--(75:1.5)--(90:1);
	\fill[color=blue!20] (0:1)--(30:1)--(60:1)--(90:1)--(120:1)--(150:1)--(180:1)--(210:1)--(240:1)--(270:1)--(300:1)--(330:1)--(0:1);
	\fill[color=white] (0,0)--(0.5,0)--(0.5,-0.5)--(0,-0.5)--(0,0);
	\draw (0:1) node{}--(30:1) node{};
	\draw (30:1) node{}--(60:1) node{};
	\draw[blue, very thick] (60:1) node{} -- (90:1) node{};
	\draw (90:1) node{}--(120:1) node{};
	\draw (120:1) node{}--(150:1) node{};
	\draw (150:1) node{}--(180:1) node{};
	\draw (180:1) node{}--(210:1) node{};
	\draw (210:1) node{}--(240:1) node{};
	\draw (240:1) node{}--(270:1) node{};
	\draw (270:1) node{}--(300:1) node{};
	\draw (300:1) node{}--(330:1) node{};
	\draw (330:1) node{}--(0:1) node{};
	\draw (60:1) node{}--(75:1.5) node{};
	\draw (90:1) node{}--(75:1.5) node{};
	\draw (0,0) node{} -- (0.5,0) node{} -- (0.5,-0.5) node{} -- (0,-0.5) node{} -- (0,0) node{};
	
	\draw[blue](75:0.8)node[texte]{$e$};
	\draw(65:1.5)node[texte]{$x$};
	\draw(-0.3,0.3)node[texte]{$t$};
	\draw(-1,1)node[texte]{$\infty$};
	\draw(0.25,-0.25)node[texte]{$\infty$};
	\draw(75:1.2)node[texte]{$f$};		
		
	\begin{scope}[shift={(4,0)}]
	\fill[color=violet!20]
		(90:1)--(60:1) to[out=15,in=120] (30:1.5) to[out=300,in=45] (0:1) to[out=30,in=315] (45:1.8) to[out=135,in=45] (90:1);
	\fill[color=blue!20] (0:1)--(30:1)--(60:1)--(90:1)--(120:1)--(150:1)--(180:1)--(210:1)--(240:1)--(270:1)--(300:1)--(330:1)--(0:1);
	\fill[color=white] (0,0)--(0.5,0)--(0.5,-0.5)--(0,-0.5)--(0,0);
		\draw (0:1) node{}--(30:1) node{};
	\fill[color=green!20]
		(60:1) to[out=15,in=120] (30:1.5) to[out=300,in=45] (0:1) -- (30:1) -- (60:1);
	\draw (0:1) node{}--(30:1) node{};
	\draw (30:1) node{}--(60:1) node{};
	\draw[blue, very thick] (60:1) node{}--(90:1) node{};
	\draw (90:1) node{}--(120:1) node{};
	\draw (120:1) node{}--(150:1) node{};
	\draw (150:1) node{}--(180:1) node{};
	\draw (180:1) node{}--(210:1) node{};
	\draw (210:1) node{}--(240:1) node{};
	\draw (240:1) node{}--(270:1) node{};
	\draw (270:1) node{}--(300:1) node{};
	\draw (300:1) node{}--(330:1) node{};
	\draw (330:1) node{}--(0:1) node{};
	\draw (60:1) node{} to[out=15,in=120] (30:1.5) to[out=300,in=45] (0:1) node{};
	\draw (0:1) node{} to[out=30,in=315] (45:1.8) to[out=135,in=45] (90:1) node{};
	\draw(35:1.35) node{} to[bend left] (0:1) node{};
	\draw(35:1.35) node{} to[bend right] (30:1) node{};
	\draw(35:1.35) node{}--(60:1) node{};
	\draw(20:1.1)node{}--(35:1.35) node{};
	\draw(20:1.1)node{}--(30:1) node{};
	\draw(20:1.1)node{}--(0:1) node{};
	\draw (0,0) node{} -- (0.5,0) node{} -- (0.5,-0.5) node{} -- (0,-0.5) node{} -- (0,0) node{};
	
	\draw[blue](75:0.8)node[texte]{$e$};
	\draw(-10:1.2)node[texte]{$x$};
	\draw(-0.3,0.3)node[texte]{$t$};
	\draw(-1,1)node[texte]{$\infty$};
	\draw(0.25,-0.25)node[texte]{$\infty$};
	\draw(0.8,1.2)node[texte]{$f$};
	\end{scope}
		
	\begin{scope}[shift={(8.5,0)}]
	\fill[color=violet!20]
	(60:1) to[out=150,in=30] (130:2) to[out=210,in=135] (180:1) to[out=120,in=225] (135:1.7) to[out=45,in=150] (90:1)--(60:1);
	\fill[color=blue!20] (0:1)--(30:1)--(60:1)--(90:1)--(120:1)--(150:1)--(180:1)--(210:1)--(240:1)--(270:1)--(300:1)--(330:1)--(0:1);
	\fill[color=white] (0,0)--(0.5,0)--(0.5,-0.5)--(0,-0.5)--(0,0);
	\draw (0:1) node{}--(30:1) node{};
	\draw (0:1) node{}--(30:1) node{};
	\draw (30:1) node{}--(60:1) node{};
	\draw[blue, very thick] (60:1) node{}--(90:1) node{};
	\draw (90:1) node{}--(120:1) node{};
	\draw (120:1) node{}--(150:1) node{};
	\draw (150:1) node{}--(180:1) node{};
	\draw (180:1) node{}--(210:1) node{};
	\draw (210:1) node{}--(240:1) node{};
	\draw (240:1) node{}--(270:1) node{};
	\draw (270:1) node{}--(300:1) node{};
	\draw (300:1) node{}--(330:1) node{};
	\draw (330:1) node{}--(0:1) node{};
	\draw (60:1) node{} to[out=150,in=30] (130:2) to[out=210,in=135] (180:1) node{};
	\draw (180:1) node{} to[out=120,in=225] (135:1.7) to[out=45,in=150] (90:1) node{};
	\draw (0,0) node{} -- (0.5,0) node{} -- (0.5,-0.5) node{} -- (0,-0.5) node{} -- (0,0) node{};
	
	\draw[blue](75:0.8)node[texte]{$e$};
	\draw(190:1.2)node[texte]{$x$};
	\draw(-0.3,0.3)node[texte]{$t$};
	\draw(1,1)node[texte]{$\infty$};
	\draw(0.25,-0.25)node[texte]{$\infty$};
	\draw(-1.2,1.4)node[texte]{$f$};
	\draw(-0.9,0.9)node[texte]{$\infty$};	
	\end{scope}
	\end{tikzpicture}
\end{center}
\caption{The three cases appearing in the peeling equations. The case in the middle corresponds to $(i,j)=(2,4)$ in the first sum. The case on the right corresponds to $i=3$ in the second.}\label{fig_peeling_cases}
\end{figure}

The main steps of the proof of Theorem~\ref{main_thm} will be as follows: in Section~\ref{subsec_solve_peeling}, we will express the solutions to the peeling equations in terms of a pair $(\Lambda, \Gamma)$ of random variables (Proposition~\ref{prop_sol_peeling}). These two variables can be thought of as Boltzmann weights on respectively the volume and the number of infinite ends. The PSHT $\TT_{\lambda}$ corresponds to the case $\Lambda=\lambda$ and $\Gamma=0$, whereas $\TT_{\star}$ corresponds to the case $\Lambda=0$ and $\Gamma=1$. In Section~\ref{subsec_degenerate_or_oneended}, using the formula obtained in Proposition~\ref{prop_sol_peeling} and the nonnegativity of the coefficients $a_v^{p_1, \dots, p_k}$, we will prove that almost surely $\Lambda=0$ or $\Gamma=0$. Finally, in Section~\ref{subsec_end_proof}, we will exclude the case $0<\Gamma<1$ using a similar strategy.

\subsection{Solving the peeling equations}\label{subsec_solve_peeling}

The ideas used in the first part of the proof will be very similar to the ideas used for the one-ended case in~\cite[Section 3]{BL19}: we will use the Hausdorff moment problem to interpret the coefficients $a_v^{1, \dots, 1}$ as the moments of certain random variables. We first notice that $T$ is characterized by these coefficients. We will use the notation $k \otimes 1$ for $1, \dots, 1$ where $1$ appears $k$ times.

\begin{lem}\label{lem_determined_by_one}
The distribution of $T$ is determined by the numbers $a_v^{k \otimes 1}$ for $v \geq 0$ and $k \geq 1$.
\end{lem}

\begin{proof}
It is sufficient to prove that for all $p_1, \dots, p_k \geq 1$ and $v \geq 0$, the coefficient $a_v^{p_1, \dots, p_k}$ is determined by $\left( a_v^{\ell \otimes 1} \right)_{\ell \geq 1, \, v \geq 0}$. We prove this by induction on $\sum_{i=1}^k (p_i-1)$. If this sum is zero, then $p_i=1$ for all $i$ so the result is immediate.
	
Moreover, if $\sum_{i=1}^k (p_i-1)>0$, assume without loss of generality $p_1 \geq 2$. Then the peeling equation can be rewritten
\[ a_v^{p_1, \dots, p_k} = a_v^{p_1-1, p_2, \dots, p_k} -2 \sum_{i=0}^{p_1-2} \sum_{j \geq 0} a_{v+j}^{p_1-1-i, p_2, \dots, p_k} \tau_j(i+1) - \sum_{i=0}^{p_1-2} a_v^{i+1, p_1-1-i,p_2, \dots, p_k}. \]
By the induction hypothesis, all the terms in the right-hand side are determined by $\left( a_v^{\ell \otimes 1} \right)_{\ell \geq 1, \, v \geq 0}$, so this is also true for the left-hand side, which proves the lemma.
\end{proof}

The next step, which follows closely~\cite[Section 3]{BL19}, is to prove that the $a_v^{k \otimes 1}$ are given by the moments of a pair of random variables.

\begin{lem}\label{lem_apv_as_moments}
There is a random variable $\left( \Lambda, \Gamma \right)$ with values in $[0, \lambda_c] \times [0,1]$ such that, for all $k \geq 1$ and $v \geq 0$, we have:
\begin{equation}\label{eqn_apv_as_moments}
a_v^{k \otimes 1} = \E \left[ \Lambda^{v} \Gamma^{k-1} \right].
\end{equation}
\end{lem}

\begin{proof}
We define the discrete derivative operators $\Delta_v$ and $\Delta_k$ by
\[ \left( \Delta_v a \right)_v^{p_1, \dots, p_k} = a^{p_1, \dots, p_k}_v - a^{p_1, \dots, p_k}_{v+1}\]
and
\[ \left( \Delta_k a \right)_v^{p_1, \dots, p_k} = a^{p_1, \dots, p_k}_v - a^{p_1, \dots, p_k,1}_{v}.\]
Note that $\left( \Delta_v a^{k \otimes 1}_v \right)$ and $\left( \Delta_k a^{k \otimes 1}_v \right)$ are respectively the discrete derivatives of $\left( a^{k \otimes 1}_v \right)$ with respect to $v$ and to $k$. Since $a_0^1=1$, by the two-dimensional Hausdorff moment problem~\cite{HS33}, it is sufficient to prove $\Delta_v^{m} \Delta_k^n a_v^{k \otimes 1} \geq 0$ for all $m,n,v \geq 0$ and $k \geq 1$. We will actually prove the following more general inequality for all $m,n,v \geq 0$ and $p_1, \dots, p_k \geq 1$:
\begin{equation}\label{eqn_derivatives_nonnegative}
\Delta_v^m \Delta_k^n a_v^{p_1, \dots, p_k} \geq 0.
\end{equation}
We will prove~\eqref{eqn_derivatives_nonnegative} by induction on $m+n$. Although we only need the case $p_1=\dots=p_k=1$ in the end, handling the general case will be necessary in the induction step.

The case $m=n=0$ is immediate since the coefficients $a_v^{p_1, \dots, p_k}$ are nonnegative. We now assume the result holds for $(m,n)$, and prove it for $(m,n+1)$. Let $v \geq 0$ and $p_1, \dots, p_k \geq 1$. Using the induction hypothesis and writing the peeling equation for $a_{v+v'}^{p_1, \dots, p_k, k' \otimes 1}$ for $0 \leq v' \leq m$ and $0 \leq k' \leq n$, we have:
\begin{align*}
0 & \leq \Delta^m_v \Delta_k^n a_v^{p_1+1, p_2, \dots, p_k}\\
&= \Delta^m_v \Delta_k^n a_v^{p_1, p_2, \dots, p_k} -2 \sum_{i=0}^{p_1-1} \sum_{j \geq 0} \tau_j(i+1) \Delta^m_v \Delta_k^n a_{v+j}^{p_1-i, p_2, \dots, p_k} - \sum_{i=0}^{p_1-1} \Delta^m_v \Delta_k^n a_v^{i+1, p_1-i, p_2, \dots, p_k}.
\end{align*}
By the induction hypothesis, every term in the two sums is nonnegative. Therefore, the last inequality remains true if we remove entirely the first sum, and keep only the term $i=0$ in the second one. We obtain
\[ 0 \leq \Delta^m_v \Delta_k^n a_v^{p_1, p_2, \dots, p_k} - \Delta^m_v \Delta_k^n a_v^{1, p_1, p_2, \dots, p_k} = \Delta^m_v \Delta_k^{n+1} a_v^{p_1, p_2, \dots, p_k}, \]
which proves the result for $(m,n+1)$. The argument to deduce the result for $(m+1,n)$ from $(m,n)$ is the same, but this time we keep only the term $i=0, j=1$ in the first sum and remove the factor $2$, exactly as in~\cite[Lemma 16]{BL19}. This proves the claim~\eqref{eqn_derivatives_nonnegative}.

Therefore, by the Hausdorff moment problem, there is a random variable $(\Lambda, \Gamma)$ with values in $[0,1]^2$ such that~\eqref{eqn_apv_as_moments} holds for all $k \geq 1$ and $v \geq 0$. To conclude, we only need to show $\Lambda \leq \lambda_c$ almost surely. For this, by the peeling equation for $v=0$, $k=1$ and $p_1=1$, we must have:
\[\E \left[ Z_1(\Lambda) \right] = \sum_{j \geq 0} \tau_j(1) \E \left[ \Lambda^j \right] = \sum_{j \geq 0} a^1_{j} \tau_j(1) \leq a_0^1 < +\infty. \]
Hence, we must have $Z_1(\Lambda) <+\infty$ a.s.. By the results of Section~\ref{subsec_combinatorics}, this means $\Lambda \leq \lambda_c$ a.s..
\end{proof}

We note that by~\eqref{eqn_defn_PSHT} and~\eqref{eqn_Cp_psht}, the PSHT $\TT_{\lambda}$ for $\lambda \in [0,\lambda_c]$ corresponds to the case where $(\Lambda,\Gamma)=(\lambda,0)$ a.s.. On the other hand, by Lemma~\ref{lem_Tstar}, the Markovian triangulation $\TT_{\star}$ corresponds to the case $(\Lambda,\Gamma)=(0,1)$ a.s.. Therefore, proving the main theorem is equivalent to showing that almost surely, we have either $\Gamma=0$ or $(\Lambda,\Gamma)=(0,1)$.

It follows from Lemmas~\ref{lem_determined_by_one} and~\ref{lem_apv_as_moments} that the coefficients $a_v^{p_1, \dots, p_k}$ are characterized by the law of $(\Lambda, \Gamma)$. Our next step is to give an explicit formula for these coefficients in terms of $(\Lambda, \Gamma)$.

We recall from Section~\ref{subsec_combinatorics} that $Z_i(\lambda)$ is the partition function of Boltzmann triangulations of the $i$-gon with Boltzmann weight $\lambda$ on the volume. For all $(\lambda, \gamma) \in [0,\lambda_c] \times [0,1]$, let $\left( C_p(\lambda,\gamma) \right)_{p \geq 1}$ be the sequence satisfying $C_1(\lambda,\gamma)=1$ and, for all $p \geq 1$:
\begin{equation}\label{eqn_peeling_Cp}
C_p(\lambda,\gamma) = C_{p+1}(\lambda,\gamma) + 2 \sum_{i=0}^{p-1} Z_{i+1}(\lambda) C_{p-i}(\lambda,\gamma) + \gamma \sum_{i=0}^{p-1} C_{i+1}(\lambda, \gamma) C_{p-i}(\lambda, \gamma).
\end{equation}
Note that this formula defines $C_p(\lambda, \gamma)$ in a nonambiguous way, since it allows to express $C_{p+1}$ using only previous terms. We note that for $\gamma=0$, we recover~\eqref{eqn_Cp_psht}, which means that $C_p(\lambda,0)=C_p^{PSHT}(\lambda)$. We also note right now that, by induction on $p$, the function $C_p(\lambda, \gamma)$ is a continuous function of $(\lambda, \gamma) \in [0,\lambda_c] \times [0,1]$. In particular, it is bounded by a constant $f(p)$.

We can now express all the coefficients $a_v^{p_1, \dots, p_k}$ in terms of $(\Lambda, \Gamma)$.

\begin{prop}\label{prop_sol_peeling}
For all $p_1, \dots, p_k \geq 1$ and $v \geq 0$, we have
\begin{equation}\label{eqn_sol_peeling}
a^{p_1,\dots,p_k}_v = \E \left[ \Lambda^{v} \Gamma^{k-1} \prod_{i=1}^k C_{p_i}(\Lambda, \Gamma) \right].
\end{equation}
In particular, for all $p,k \geq 1$ and $v \geq 0$, we have
\begin{equation}\label{eqn_apv_intermsof_C}
a^{p, (k-1) \otimes 1}_v = \E \left[ \Lambda^v \Gamma^{k-1} C_p(\Lambda, \Gamma) \right].
\end{equation}
\end{prop}

\begin{proof}
For all $(\lambda, \gamma) \in [0,\lambda_c] \times [0,1]$, we define
\[a^{p_1,\dots,p_k}_v(\lambda, \gamma) = \lambda^{v} \gamma^{k-1} \prod_{i=1}^k C_{p_i}(\lambda, \gamma).\]
Using~\eqref{eqn_peeling_Cp}, it is easy to check that $\left( a_v^{p_1, \dots, p_k}(\lambda, \gamma) \right)$ is a solution to the peeling equations, with $a_v^{k \otimes 1}(\lambda, \gamma)=\lambda^v \gamma^{k-1}$. By linearity of the peeling equations, it follows that the right-hand side of~\eqref{eqn_sol_peeling} is also a solution. Therefore, both sides of~\eqref{eqn_sol_peeling} are solutions to the peeling equations. By Lemma~\ref{lem_apv_as_moments}, they coincide for $p_1=\dots=p_k=1$. Therefore, by Lemma~\ref{lem_determined_by_one}, both sides coincide everywhere. Finally,~\eqref{eqn_apv_intermsof_C} is obtained by taking $p_2=\dots=p_k=1$.
\end{proof}

\subsection{A Markovian triangulation is either degenerate or one-ended}\label{subsec_degenerate_or_oneended}

The next step of the proof is to show that almost surely, either $\Gamma=0$ (which can be interpreted as $T$ being one-ended), or $\Lambda=0$ (which can be interpreted as $T$ being degenerate, i.e. with infinite vertex degrees). We will do so by proving that if this is not true, then one of the coefficients $a_v^{p,(k-1) \otimes 1}$ is negative. We first consider the case where $(\Lambda, \Gamma)$ is deterministic.

\begin{prop}\label{prop_cnegative_deterministic}
If $\lambda \in (0,\lambda_c]$ and $\gamma \in (0,1]$, then there is $p \geq 1$ such that $C_p(\lambda,\gamma) <0$.
\end{prop}

\begin{proof}
We fix $\lambda, \gamma>0$ throughout the proof and write $\mathcal{C}_{\lambda, \gamma}(x)=\sum_{p \geq 1} C_p(\lambda,\gamma) x^p$.
Then the recursion~\eqref{eqn_peeling_Cp} on $C_p$ becomes
\begin{equation}\label{eqn_genfunction_Cp}
\mathcal{C}_{\lambda, \gamma}(x)=\frac{1}{x} \left( \mathcal{C}_{\lambda, \gamma}(x)-x \right) + \frac{2}{x} \mathcal{Z}_{\lambda}(x) \mathcal{C}_{\lambda, \gamma}(x) + \frac{\gamma}{x} \mathcal{C}_{\lambda, \gamma}(x)^2,
\end{equation}
where $\mathcal{Z}(x)=\sum_{p \geq 1} Z_p(\lambda) x^p$ is given by~\eqref{eqn_computation_Z}.
%If $\lambda=0$, using $Z_1(0)=0$, the recursion is just
%\begin{equation}\label{eqn_generatingfunction_lambdazero}
%\mathcal{C}_{\lambda, \gamma}(x)=\frac{1}{x} \left( \mathcal{C}_{\lambda, \gamma}(x)-x \right) + \frac{\gamma}{x} \mathcal{C}_{\lambda, \gamma}(x)^2.
%\end{equation}
After solving the quadratic equation, we find
\[\mathcal{C}_{\lambda, \gamma}(x)=\frac{1}{2\gamma} \left( \sqrt{\left( 1-\frac{x}{\sqrt{1+8h}} \right)^2 \left( 1-\frac{4h}{\sqrt{1+8h}}x \right) +4\gamma x} -\left( 1-\frac{x}{\sqrt{1+8h}} \right) \sqrt{1-\frac{4h}{\sqrt{1+8h}}x} \right),  \]
where $h \in [0,1/4]$ is given by~\eqref{eqn_defn_h}.

Now assume that all the $C_p(\lambda, \gamma)$ are nonnegative. By the Pringsheim theorem, the radius of convergence of $\left( C_p(\lambda,\gamma) \right)$ is equal to the first nonnegative singularity of $\mathcal{C}_{\lambda,\gamma}$. On the other hand,
this function has a singularity at $x=\frac{\sqrt{1+8h}}{4h}$. Moreover, for $x < \frac{\sqrt{1+8h}}{4h}$, the inside of the first square root is clearly positive, so the radius of convergence of $\left( C_p(\lambda, \gamma) \right)$ is $\frac{\sqrt{1+8h}}{4h}$. However, if $h<1/4$ (i.e. $\lambda<\lambda_c$), then we can compute:
\[ \lim_{x \to \frac{\sqrt{1+8h}}{4h}} \mathcal{C}_{\lambda,\gamma}'(x) = -\infty,\]
which is a contradiction. Similarly, in the critical case $h=1/4$, the radius of convergence is $\sqrt{3}$ and
\[\lim_{x \to \sqrt{3}} \mathcal{C}_{\lambda,\gamma}''(x) = -\infty.\]
\end{proof}

\begin{prop}\label{prop_cnegative_random}
With the notation of Lemma~\ref{lem_apv_as_moments}, we have almost surely either $\Lambda=0$ or $\Gamma=0$.
\end{prop}

\begin{proof}
Assume that $\P \left( \Lambda > 0, \Gamma > 0 \right)>0$. Since the coefficients $a_v^{p_1, \dots, p_k}$ must be nonnegative, by~\eqref{eqn_apv_intermsof_C}, it is sufficient to find $p,k \geq 1$ and $v \geq 0$ such that
\begin{equation}\label{eqn_negative_expectation}
\E \left[ \Lambda^v \Gamma^{k-1} C_p(\Lambda, \Gamma) \right] <0.
\end{equation}

Let $K$ be the compact support of the law of $(\Lambda, \Gamma)$. We first claim that there is a coefficient $\alpha \geq 1$ such that the quantity $\lambda \times \gamma^{\alpha}$ has a unique maximizer in $K$. Indeed, we write
\[ \log K = \left\{ \left( \log \lambda, \log \gamma \right) | \lambda>0, \gamma>0, (\lambda, \gamma) \in K \right\}.\]
This is nonempty by assumption, so the convex hull $\overline{\log K}$ of $\log K$ is a nonempty convex subset of $\left( \mathbb{R}^- \right)^2$. Its boundary $\partial \overline{\log K}$ contains at most a countable number of nontrivial segments. Therefore, there is a vector $(1,\alpha)$ with $\alpha \geq 1$ such that $\partial \overline{\log K}$ contains no segment orthogonal to $(1,\alpha)$. This implies that $x+\alpha y$ has a unique maximizer in $\log K$, so $\lambda \times \gamma^{\alpha}$ has a unique maximizer in $K$. We denote by $(\lambda_0,\gamma_0)$ this maximizer.

By Proposition~\ref{prop_cnegative_deterministic}, there is $p_0 \geq 1$ such that $C_{p_0} (\lambda_0, \gamma_0) <0$. The idea will be that if we take $p=p_0$ and $k=\alpha v$ and let $v$ go to infinity, then the mass of the expectation~\eqref{eqn_negative_expectation} is concentrated close to $(\Lambda, \Gamma)=(\lambda_0,\gamma_0)$.

More precisely, we denote by $B_{\eps}(\lambda_0, \gamma_0)$ the ball of radius $\eps$ around $(\lambda_0,\gamma_0)$ in $\R^2$. By continuity (see the remark just before Proposition~\ref{prop_sol_peeling}), we fix $\eps>0$ such that, for all $(\lambda, \gamma) \in B_{\eps}(\lambda_0,\gamma_0)$, we have
\[C_{p_0}(\lambda, \gamma) \leq -\eps.\]
By definition of $(\lambda_0, \gamma_0)$ as a unique maximizer, there is $\delta>0$ such that, if $(\lambda, \gamma) \in K \backslash B_{\eps}(\lambda_0, \gamma_0)$, then $\lambda \gamma^{\alpha} < \lambda_0 \gamma_0^{\alpha}-\delta$.
We can now rewrite the expectation~\eqref{eqn_negative_expectation} for $p=p_0$ as
\begin{equation}\label{eqn_expectation_splitin2}
\E \left[ \Lambda^v \Gamma^{k-1} C_{p_0}(\Lambda, \Gamma) \mathbbm{1}_{\Lambda \Gamma^{\alpha} \geq \lambda_0 \gamma_0^{\alpha}-\delta} \right] + \E \left[ \Lambda^v \Gamma^{k-1} C_{p_0}(\Lambda, \Gamma) \mathbbm{1}_{\Lambda \Gamma^{\alpha} < \lambda_0 \gamma_0^{\alpha}-\delta} \right].
\end{equation}
By definition of $\delta$, if the inside of the first expectation is nonzero, then $(\Lambda, \Gamma) \in B_{\eps}(\lambda_0,\gamma_0)$, so $C_{p_0}(\Lambda, \Gamma)<0$. Therefore, the inside of the first expectation is nonpositive, so we can bound it from above by
\[ \E \left[ \Lambda^v \Gamma^{k-1} C_{p_0}(\Lambda, \Gamma) \mathbbm{1}_{\Lambda \Gamma^{\alpha} \geq \lambda_0 \gamma_0^{\alpha}-\delta/2} \right]. \]
We now take $k = \lfloor \alpha v \rfloor+1$, so that $\Lambda^v \Gamma^{k-1} \geq (\Lambda \Gamma^{\alpha})^v$. Then this first term can be bounded from below by
\[ -\eps \left( \lambda_0 \gamma_0^{\alpha}-\frac{\delta}{2} \right)^v \P \left( \Lambda \Gamma^{\alpha} \geq \lambda_0 \gamma_0^{\alpha}-\delta/2 \right), \]
where the probability is positive by definition of the support $K$.

We move on to the second term of~\eqref{eqn_expectation_splitin2}, again with $k = \lfloor \alpha v \rfloor+1$. Since $\alpha \geq 1$, we have $\Lambda^v \Gamma^{\lfloor \alpha v \rfloor} \leq \left( \Lambda \Gamma^{\alpha} \right)^{(v-1)}$, so the second term can be bounded by
\[ \E \left[ \left( \Lambda \Gamma^{\alpha} \right)^{v-1} \left| C_p(\Lambda, \Gamma) \right| \mathbbm{1}_{\Lambda \Gamma^{\alpha} < \lambda_0 \gamma_0^{\alpha}-\delta} \right] \leq \left( \lambda_0 \gamma_0^{\alpha} -\delta \right)^{v-1} f(p_0), \]
where $f(p_0)$ is a bound on $|C_{p_0}(\lambda, \gamma)|$ for all $\lambda, \gamma$ (see the discussion right before Proposition~\ref{prop_sol_peeling}). Combining the last two displays and letting $v \to +\infty$, we get the result.
\end{proof}

\subsection{End of the proof: degenerate triangulations}\label{subsec_end_proof}

Finally, we have to treat the cases $\Lambda=0$ and $\Gamma=0$. The second one corresponds to the one-ended PSHT, so we need to focus on the first. We will prove that it is not possible to have $\Lambda=0$ but $0<\Gamma<1$. The proof will be very similar to the previous argument (Proposition~\ref{prop_cnegative_random}), with the difference that this time, it is not sufficient to look at the maximum of the support of $\Gamma$, since it is possible that $\Gamma=1$. Therefore, we will first argue that either $\Gamma=1$, or $0 \leq \Gamma \leq \frac{1}{2}$, and then consider the maximum of the support of $\Gamma$ minus $1$. As in Section~\ref{subsec_degenerate_or_oneended}, we will start with the case where $(\lambda,\gamma)$ is deterministic.

\begin{lem}\label{lem_cnegative_deter_lambdazero}
	For all $0<\gamma<1$, there is $p \geq 1$ such that $C_p(0,\gamma)<0$. Moreover, if $\frac{1}{2}<\gamma<1$, we can take $p=3$.
\end{lem}

\begin{proof}
We first consider the case $\lambda=0$ in the induction~\eqref{eqn_peeling_Cp}. Using $Z_{i+1}(0)=0$, we obtain	
\[C_p(0,\gamma) = C_{p+1}(0,\gamma) + \gamma \sum_{i=0}^{p-1} C_{i+1}(0, \gamma) C_{p-i}(0, \gamma). \]
In particular, using the cases $p=1$ and $p=2$, since $C_1(0,\gamma)=1$, we must have $C_2(0,\gamma)=1-\gamma$ and
\begin{equation}\label{eqn_p_equal_to_3}
C_3(0,\gamma)=(1-\gamma)(1-2\gamma).
\end{equation}
In particular, we have $C_3(0,\gamma)<0$ as soon as $\frac{1}{2} < \gamma <1$, which proves the second point.

On the other hand, we recall that $\mathcal{C}_{\lambda,\gamma}(x)=\sum_{p \geq 1} C_p(\lambda,\gamma) x^p$. By solving~\eqref{eqn_genfunction_Cp} and using $\mathcal{Z}_0(x)=0$, we get:
\[ \mathcal{C}_{0,\gamma}(x)= \begin{cases}
\frac{1}{2\gamma} \left( \sqrt{(1-x)^2+4\gamma x}-(1-x) \right) & \mbox{if $\gamma>0$,} \label{eqn_gf_lambdazero_explicit}\\
\frac{x}{1-x} & \mbox{if $\gamma=0$.}
\end{cases}
\]
In particular, the generating function $\mathcal{C}_{0,\gamma}$ has at most two singularities, which are conjugate complex numbers of modulus $1$. Therefore, the radius of convergence of $\left(  C_p(0,\gamma) \right)_{p \geq 1}$ is at least one.

On the other hand, we can compute the third derivative for $0<x<1$:
\[ \sum_{p \geq 3} p(p-1)(p-2) C_p(0,\gamma) x^{p-3} = \mathcal{C}_{0,\gamma}'''(x) = \frac{6(1-\gamma)(1-2\gamma-x)}{\left( (1-x)^2+4\gamma x \right)^{5/2}}.\]
If $0<\gamma<1$ and $x$ is close enough to $1$, this is negative, which proves the lemma.
\end{proof}

We can now finish the proof.

\begin{proof}[Proof of Theorem~\ref{main_thm}]
As noted right after the proof of Lemma~\ref{lem_apv_as_moments}, all we need to prove is that either $\Gamma=0$ or $(\Lambda, \Gamma)=(0,1)$. Given Proposition~\ref{prop_cnegative_random}, all we have left to prove is that $\P \left( 0<\Gamma<1 \right)=0$.

Now let $p \geq 1$ and $k \geq 1$. By~\eqref{eqn_apv_intermsof_C}, we have
\[a_0^{p, k \otimes 1} = \E \left[ \Gamma^{k} C_p(\Lambda, \Gamma) \right] = \E \left[ \Gamma^{k} C_p(0, \Gamma) \right],\]
where the last equality comes from the fact that if $\Gamma^{k} C_p(\Lambda, \Gamma) \ne 0$, then $\Gamma>0$ so $\Lambda=0$ by Proposition~\ref{prop_cnegative_random}. It is sufficient to prove that if $0<\Gamma<1$ with positive probability, then we can find $p,k \geq 1$ such that $\E \left[ \Gamma^{k} C_p(0, \Gamma) \right]<0$.

For $p=3$, by~\eqref{eqn_p_equal_to_3} we get
\[a_0^{3, k \otimes 1} = \E \left[ \Gamma^{k} (1-\Gamma)(1-2\Gamma) \right].\]
The quantity in the expectation is negative for $\frac{1}{2}<\Gamma<1$ and vanishes for $\Gamma=1$, so if $\P \left( \frac{1}{2} < \Gamma < 1 \right)>0$, then we have $a_0^{3, k \otimes 1}<0$ for $k$ large enough, which is not possible. Therefore, the support of $\Gamma$ is included in $\left[ 0,\frac{1}{2} \right] \cup \{1\}$.

Finally, we denote by $\gamma_0$ the maximum of the intersection of the support of $\Gamma$ with the interval $\left[ 0,\frac{1}{2} \right]$. If $\gamma_0>0$, by Lemma~\ref{lem_cnegative_deter_lambdazero}, there is $p_0 \geq 1$ such that $C_{p_0}(0,\gamma_0)<0$. Exactly in the same way as in the proof of Proposition~\ref{prop_cnegative_random}, using the fact that $C_{p_0}(0,\gamma)$ is a continuous function of $\gamma$, we deduce that $\E \left[ \Gamma^{k} C_{p_0}(0, \Gamma) \right]<0$ for $k$ large enough. Therefore, we must have $\gamma_0=0$, which concludes the proof.
\end{proof}

\section{Applications to the convergence of finite models to the UIPT}

\subsection{Triangulations with defects}

\begin{proof}[Proof of Corollary~\ref{cor_triang_with_defects}]
We first argue that any subsequential limit of $M_{\mathbf{f}^n}$ for $d_{\loc}^*$ is almost surely a triangulation.

For this, let $m_n$ be a map with face degrees given by $\mathbf{f}^n$, and recall that $|\mathbf{f}^n|$ is the total number of edges of $m_n$. Let $D(m_n)$ be the set of edges of $m_n$ incident to a face which is not a triangle. By our assumption, we have
\[ \# D(m_n) \leq \sum_{j \ne 3} j f_j^n = o \left( |\mathbf{f}^n| \right). \]
Moreover, for all $r \geq 0$, let $E_r(m_n)$ be the set of edges $e$ such that $B_r^*(m_n,e)$ is not a triangulation with holes. In other words $E_r(m_n)$ is the set of edges $e$ at dual distance at most $r$ from an edge of $D(m_n)$. If $e \in E_r(m_n)$, let $d^e$ be the edge of $D(m_n)$ which is the closest to $e$ for the dual distance (if it is not unique, pick one arbitrarily). Let also $\gamma^e$ be a shortest dual path from $d^e$ to $e$. By minimality of $d^e$, the path $\gamma^e$ is a non-backtracking dual path of length at most $r+1$ containing only triangles. Hence, for each $d^e$, the path $\gamma^e$ can take at most $2^{r+2}$ values. It follows that, for all $r \geq 0$,
\[ \# E_r(m_n) \leq 2^{r+2} \# D(m_n) = o \left( |\mathbf{f}^n| \right). \]
Since $M_{\mathbf{f}^n}$ is invariant under uniform rerooting, this implies that with probability $1-o(1)$, all the internal faces in $B_r^* \left( M_{\mathbf{f}^n} \right)$ are triangles. This is true for all $r \geq 0$, so $M_{\mathbf{f}^n}$ is tight for $d_{\loc}^*$ and any subseqential limit is a triangulation.

Now let $T$ be such a subsequential limit. First $B_r^* \left( M_{\mathbf{f}^n} \right)$ is planar for all $r,n \geq 0$, so $T$ must be planar. Moreover $|\mathbf{f}^n| \to +\infty$, so $T$ is infinite. Finally, for any triangulation $t$ with $k$ holes of perimeters $p_1, \dots,p_k$ and any $n \geq 0$, the probability $\P \left( t \subset M_{\mathbf{f}^n} \right)$ only depends on $n$, on $p_1, \dots, p_k$ and on the inner volume of $t$. Indeed, this probability is given by the number of ways to fill each hole $h_i$ of $t$ with a map $m_i$ of the $p_i$-gon, in such a way that $t \cup \bigcup_{i=1}^k m_i$ has face degrees prescribed by $\mathbf{f}^n$. By letting $n \to +\infty$, it follows that $T$ is Markovian, so by Theorem~\ref{main_thm} $T$ is of the form $\TT_{\Lambda}$, where $\Lambda$ is a random variable with values in $[0,\lambda_c] \cup \{ \star \}$.

Finally, we prove $\Lambda=\lambda_c$ almost surely by considering the mean inverse degree: by the Euler formula, the number of edges and vertices of $M_{\mathbf{f}_n}$ are respectively $|\mathbf{f}^n|$ and $2+\sum_{j \geq 1} (j-2)f_j^n$. Let $\rho_n$ be the root vertex of $M_{\mathbf{f}^n}$. Since $M_{\mathbf{f}^n}$ is invariant under rerooting on a uniform oriented edge, we have
\[ \E \left[ \frac{1}{\deg_{M_{\mathbf{f}^n}}(\rho_n)} \right] = \frac{2+\sum_{j \geq 1} (j-2)f_j^n}{2 |\mathbf{f}^n|} \xrightarrow[n \to +\infty]{} \frac{1}{6}. \]
Moreover, the inverse degree of the root vertex is a bounded, continuous function for $d_{\loc}^*$, so we must have
\begin{equation}\label{eqn_mean_degree}
\E \left[ \frac{1}{\deg_{\TT_{\Lambda}}(\rho)} \right]=\frac{1}{6}.
\end{equation}
For $\lambda \in [0,\lambda_c] \cup \{\star\}$, let $d(\lambda)$ be the expected inverse degree of the root vertex in $\TT_{\lambda}$. It is proved in~\cite[Proposition 20]{BL19} that for $\lambda \in (0,\lambda_c]$, we have $d(\lambda) \leq \frac{1}{6}$ with equality if and only if $\lambda=\lambda_c$. Moreover, it is immediate that $d(0)=d(\star)=0$ (all the vertices have infinite degree), so~\eqref{eqn_mean_degree} implies $\Lambda=\lambda_c$ and $T$ is the UIPT. We have proved that the UIPT is the only subsequential limit of $M_{\mathbf{f}^n}$, so $M_{\mathbf{f}^n} \to \TT_{\lambda_c}$ in distribution for $d_{\loc}^*$. Since the UIPT has finite vertex degrees, this implies convergence for $d_{\loc}$ by Lemma~\ref{lem_dual_convergence}.
\end{proof}

\subsection{Triangulations with moderate genus}

\begin{proof}[Proof of Corollary~\ref{cor_moderate_genus}]
The only part of the proof which differs significantly from the proof of Corollary~\ref{cor_triang_with_defects} is the proof that any subsequential limit of $T_{n,g_n}$ for $d_{\loc}^*$ is planar. If we admit this, since tightness for $d_{\loc}^*$ is immediate, any subsequential limit must be of the form $\TT_{\Lambda}$ by Theorem~\ref{main_thm}. Moreover, by the Euler formula, we have
\[ \E\left[ \frac{1}{\deg_{T_{n,g_n}}(\rho)} \right] = \frac{2-2g_n+n}{6n} \xrightarrow[n \to +\infty]{} \frac{1}{6} \]
and we conclude $\Lambda=\lambda_c$ in the same way as in Corollary~\ref{cor_triang_with_defects}.

To prove that any subsequential limit of $T_{n,g_n}$ for $d^*_{\loc}$ is planar, we need to prove that for all $r \geq 0$, the dual ball $B_r^*(T_{n,g_n})$ is planar with probability $1-o(1)$ as $n \to +\infty$. For all $n \geq 0$, let $t_n$ be a triangulation with $2n$ faces and genus $g_n$. For all $r,n \geq 0$, we denote by $NP_r(t_n)$ the set of edges $e \in t_n$ such that the dual ball $B_r^*(t_n;e)$ of radius $r$ around $e$ is not planar. We will prove that $\# NP_r(t_n)=o(n)$ uniformly in the choice of $t_n$, which is sufficient by invariance of $T_{n,g_n}$ under uniform rerooting.

For this, note that the number of edges at dual distance at most $2r$ from an edge $e$ is bounded by $3^{2r+1}$. Therefore, for any edge $e \in NP_r(t_n)$, there are at most $3^{2r+1}$ edges $e' \in NP_r(t_n)$ such that the balls $B_r^*(t_n;e)$ and $B_r^*(t_n;e')$ intersect. Therefore, we can find a subset $\widetilde{NP}_r(t_n)$ of $NP_r(t_n)$ of size at least $\frac{1}{3^{2r+1}} \# NP_r(t_n)$ such that the balls $B_r^*(t_n ; e)$ for $e \in \widetilde{NP}_r(t_n)$ are edge-disjoint. Each of these balls have genus at least $1$, so
\[¨ \frac{1}{3^{2r+1}} \# NP_r(t_n) \leq \# \widetilde{NP}_r(t_n) \leq g_n=o(n).\]
It follows that $\# NP_r(t_n)=o(n)$, which proves our claim.
\end{proof}

\subsection{Triangulations with high temperature Ising model}

\begin{proof}[Proof of Corollary~\ref{cor_Ising}]
We first note that the distance $d_{\loc}$ extends in a natural way to face-coloured triangulations by setting
\[d_{\loc}\left( (T,\sigma),(T', \sigma') \right) = \left( 1+\min \{ r \geq 0 | \left( B_r(T), \sigma|_{B_r(T)} \right) \ne \left( B_r(T'), \sigma'|_{B_r(T')} \right) \}\right)^{-1},\]
and similarly for $d_{\loc}^*$.
If $T$ is a random face-coloured triangulation and $t$ is triangulation with holes equipped with a colouring $\sigma$ of its internal face, we will write $(t,\sigma) \subset T$ for the event that $t \subset T$ and the colouring of the faces of the neighbourhood of the root of $T$ isomorphic to $t$ agrees with $\sigma$.

We first notice that changing the colour of one face only changes the probability for a face-coloured triangulation to be picked by $T_n[\beta_n]$ by at most a factor $e^{3 |\beta_n|}$. Therefore, let $t$ be a triangulation with holes, and let $\sigma, \sigma'$ be two colourings of its internal faces which differ on only one face. Then we have
\[ e^{-3|\beta_n|} \P \left( (t,\sigma) \subset T_n[\beta_n] \right) \leq \P \left( (t,\sigma') \subset T_n[\beta_n] \right) \leq e^{3|\beta_n|} \P \left( (t,\sigma) \subset T_n[\beta_n] \right).\]
Let $T$ be a subsequential limit of $T_n[\beta_n]$ for $d_{\loc}^*$. By letting $n \to +\infty$ in the last display, we have
\[ \P \left( (t,\sigma) \subset T_n[\beta_n] \right) = \P \left( (t,\sigma') \subset T_n[\beta_n] \right).\]
This implies that the probability $\P \left( (t,\sigma) \subset T_n[\beta_n] \right)$ only depends on $t$ and not on $\sigma$, so
\[ \P \left( (t,\sigma) \subset T \right) = \frac{1}{2^{\# \mathrm{Internal \, faces (t)}}} \P \left( t \subset T \right) \]
for all $t$ and $\sigma$. This means that conditionally on the triangulation $T$, the colours of its faces are just given by Bernoulli face percolation with parameter $\frac{1}{2}$.

On the other hand, for all $n \geq 0$, the probability $\P \left( (t,\sigma) \subset T_n[\beta_n] \right)$ only depends on the perimeters and inner volume of $t$ and on the colours of the internal faces of $t$ which share an edge with a hole. By letting $n$ go to infinity, the same is true for the probability $\P \left( (t,\sigma) \subset T \right)$. But since it does not depend on $\sigma$, this probability actually only depends on the perimeters and inner volume of $t$, so $T$ is Markovian. Since $T$ is also planar, it must be of the form $\TT_{\lambda}$, and we conclude using the mean inverse degree in the same way as in Corollaries~\ref{cor_triang_with_defects} and~\ref{cor_moderate_genus}.
\end{proof}

\section{Large deviations for pattern occurences in uniform triangulations}

The goal of this section is to prove Theorem~\ref{thm_large_deviations}. We recall that if $t_0$ is a finite triangulation with one or several holes, then $\occ_{t_0}(t)$ is the number of occurences of $t_0$ in a triangulation of the sphere $t$. More precisely $\occ_{t_0}(t)$ is the number of oriented edges $\vec{e}$ of $t$ such that $t_0 \subset (t;\vec{e})$, where $(t;\vec{e})$ stands for the triangulation $t$, rerooted at $\vec{e}$.

\begin{proof}[Proof of Theorem~\ref{thm_large_deviations}]
We fix a finite triangulation $t_0$ with one or several holes. Let $\left( \beta_n \right)_{n \geq 1}$ be a sequence of nonnegative numbers such that $\beta_n \to 0$. We denote by $T_n^{(\beta_n)}$ a random triangulation of the sphere with $2n$ faces such that $\P \left( T_n^{(\beta_n)}=t \right)$ is proportional to
\[\exp \left( \beta_n \occ_{t_0}(t) \right).\]
The proof consists of first showing that $T_n^{(\beta_n)}$ converges locally to the UIPT. This means that it is not possible to increase significantly $\occ_{t_0}(T_n)$ by "twisting" the uniform measure in a subexponential way. Therefore, exponential factors are necessary, so triangulations with much more occurences of $t_0$ are exponentially rare.

More precisely, as in the proof of Corollary~\ref{cor_Ising}, in order to show that $T_n^{(\beta_n)}$ converges to the UIPT, the only non-trivial point is to prove that any subsequential limit $T$ of $T_n^{(\beta_n)}$ is Markovian. For this, let $t_1$ and $t_2$ be two finite triangulations with the same hole perimeters and the same inner volume. If $t$ is a triangulation of the sphere such that $t_1 \subset t$, we denote by $\Phi(t)$ the triangulation obtained by replacing $t_1$ by $t_2$ in the neighbourhood of the root. Then $\left|\occ_{t_0}(t)-\occ_{t_0}(\Phi(t)) \right|$ is bounded by a constant $c$ depending only on $t_0$, $t_1$ and $t_2$. It follows that
\[ e^{-c \beta_n} \P \left( T_n^{(\beta_n)}=t \right) \leq \P \left( T_n^{(\beta_n)}=\Phi(t) \right) \leq e^{c \beta_n} \P \left( T_n^{(\beta_n)}=t \right). \]
By summing over $t$, we get
\[ e^{-c \beta_n} \P \left( t_1 \subset T_n^{(\beta_n)} \right) \leq \P \left( t_2 \subset T_n^{(\beta_n)} \right) \leq e^{c \beta_n} \P \left( t_1 \subset T_n^{(\beta_n)} \right). \]
Letting $n \to +\infty$, we obtain $\P \left( t_1 \subset T \right)=\P \left( t_2 \subset T \right)$, so any subsequential limit of $T_n^{(\beta_n)}$ is Markovian. Since it must also be planar and have mean degree $6$, this implies convergence to the UIPT by Theorem~\ref{main_thm}.

We now recall that $T_n$ stands for a uniform triangulation of the sphere with $2n$ faces and write $X_n=\occ_{t_0}(T_n)$ and $X_n^{(\beta_n)}=\occ_{t_0}(T_n^{(\beta_n)})$. We also write $p_0=\P \left( t_0 \subset \TT_{\lambda_c} \right)$. By invariance of $T_n$ and $T_n^{(\beta_n)}$ under uniform rerooting and convergence of both models to the UIPT, we have
\[ \frac{1}{6n} \E \left[ X_n \right] = \P \left( t_0 \subset T_n \right) \xrightarrow[n \to +\infty]{} p_0, \]
\[ \frac{1}{6n} \E \left[ X_n^{(\beta_n)} \right] = \P \left( t_0 \subset T_n^{(\beta_n)} \right) \xrightarrow[n\to +\infty]{} p_0. \]
We now fix $\eps>0$. The last convergence implies
\begin{equation}\label{eqn_xnbeta_small}
\P \left( X_n^{(\beta_n)} \leq  6 p_0 n + \eps n \right) \geq 1-\frac{\E \left[ X_n^{(\beta_n)} \right]}{6p_0 n+\eps n} \xrightarrow[n \to +\infty]{} \frac{\eps}{6p_0+\eps} >0.
\end{equation}
On the other hand, by definition we have
\begin{align*}
\P \left( X_n^{(\beta_n)} \leq 6 p_0 n + \eps n \right) &= \frac{\E \left[ e^{\beta_n X_n} \mathbbm{1}_{X_n \leq 6 p_0 n + \eps n} \right]}{\E \left[ e^{\beta_n X_n} \right]}\\
& \leq \frac{\E \left[ e^{\beta_n X_n} \mathbbm{1}_{X_n \leq 6 p_0 n + \eps n} \right]}{\E \left[ e^{\beta_n X_n} \mathbbm{1}_{X_n \geq 6 p_0 n + 2 \eps n} \right]}\\
& \leq \frac{e^{\beta_n(6p_0n+\eps n)} \P \left( X_n \leq 6p_0n+\eps n \right)}{e^{\beta_n(6p_0n+2\eps n)} \P \left( X_n \geq 6p_0n+2\eps n \right)}\\
& \leq \frac{e^{-\eps \beta_n n}}{\P \left( X_n \geq 6p_0n+2\eps n \right)}.
\end{align*}
Therefore, we have
\[ \P \left( X_n \geq 6p_0n+2\eps n \right) \leq \frac{e^{-\eps \beta_n n}}{\P \left( X_n^{(\beta_n)} \leq 6 p_0 n + \eps n \right)} = O \left( e^{-\eps \beta_n n} \right) \]
by~\eqref{eqn_xnbeta_small}. Since $(\beta_n)$ can be any sequence going to $0$, we have proved that the sequence
\[ \left( -\frac{1}{n} \log \P \left( X_n \geq 6p_0n+2\eps n \right) \right)_{n \geq 0} \]
is eventually larger than any sequence which goes to $0$. Hence, it is bounded from below by a positive constant, which means that $\P \left( X_n \geq 6p_0n+2\eps n \right)$ decays exponentially in $n$. The same bound for $\P \left( X_n \leq 6p_0n-2\eps n \right)$ is proved in the same way, but by biasing by $e^{-\beta_n X_n}$ instead of $e^{\beta_n X_n}$.
\end{proof}

\section{Extensions and conjectures}\label{sec_conj}

There are two natural ways to try to extend Theorem~\ref{main_thm}: the first one is to consider more general classes of planar maps than triangulations, and the second is to remove the planarity assumption. We expect that an analog of Theorem~\ref{main_thm} for more general models of maps should be true. In this context, a Markovian map would be an infinite random map $M$ where $\P \left( m \subset M \right)$ depends only on the perimeters of the holes of $m$, and the family of degrees of its internal faces, as in~\cite[Theorem 4]{BL20}. The extension of the proof of Theorem~\ref{main_thm} does not seem straighforward, since the partition functions are not as explicit as for triangulations (only the pointed partition functions are explicit). Moreover, the extension of the Infinite Boltzmann Planar Maps of~\cite{BL20} to the infinite-degree case would be more difficult, because different ways to let Boltzmann weights go to zero give rise to different degenerate objects.

On the other hand, the non-planar case seems difficult. The natural definition of a Markovian triangulation in the nonplanar case is that $\P \left( t \subset T \right)$ depends only on the volume of $t$, on the perimeters of its holes \emph{and on its genus}. We conjecture the following.

\begin{conj}
${}$
\begin{itemize}
\item
There are Markovian infinite triangulations which are not planar.
\item
There is no Markovian infinite nonplanar triangulation with only finite vertex degrees.
\end{itemize}
\end{conj}

The motivation for the first point is that uniform triangulations with $2n$ faces and genus $g_n=\frac{n}{2}-o(n)$ should have Markovian local limits for $d_{\loc}^*$. In particular, for all $k \geq 1$, we should obtain a Markovian limit with exactly $k$ vertices for $g_n=\frac{n+2-k}{2}$.
On the other hand, proving the second point of this conjecture would remove the necessity to use the Goulden--Jackson equation in~\cite{BL19}. However, it seems hopeless to prove this second point using the same sketch as for Theorem~\ref{main_thm}, at least for the following two reasons:
\begin{itemize}
	\item the generating function for triangulations with higher genus is not explicitely known;
	\item when we write down the peeling equations for $a_{v,g}^{k \otimes 1}$ and $a_{v,g}^{(k+1) \otimes 1}$ to adapt the proof of Lemma~\ref{lem_apv_as_moments}, the lists of terms appearing in both equations are different because two holes can now be filled with the same component. Therefore, it is not easy anymore to take the discrete derivative $\Delta_k$ of the peeling equation with respect to the number of holes.
\end{itemize}

\bibliographystyle{abbrv}
\bibliography{bibli}

\begin{thebibliography}{10}

\bibitem{AMS20}
M.~Albenque, L.~Ménard, and G.~Schaeffer.
\newblock Local convergence of large random triangulations coupled with an
  {I}sing model.
\newblock {\em Transactions of the American Mathematical Society}, 374:1, 10
  2020.

\bibitem{Ang03}
O.~Angel.
\newblock Growth and percolation on the uniform infinite planar triangulation.
\newblock {\em Geom. Funct. Anal.}, 13(5):935--974, 2003.

\bibitem{AR13}
O.~Angel and G.~Ray.
\newblock Classification of half planar maps.
\newblock {\em Ann. Probab.}, 43(3):1315--1349, 2015.

\bibitem{AS03}
O.~Angel and O.~Schramm.
\newblock Uniform infinite planar triangulations.
\newblock {\em Comm. Math. Phys.}, 241(2-3):191--213, 2003.

\bibitem{BS13}
J.~E. Bj{\"o}rnberg and S.~O. Stefansson.
\newblock Recurrence of bipartite planar maps.
\newblock {\em Electron. J. Probab.}, 19(31):1--40, 2014.

\bibitem{B16}
T.~Budzinski.
\newblock The hyperbolic {B}rownian plane.
\newblock {\em Probability Theory and Related Fields}, 171(1):503--541, Jun
  2018.

\bibitem{BL19}
T.~{Budzinski} and B.~{Louf}.
\newblock {Local limits of uniform triangulations in high genus}.
\newblock {\em Inventiones Mathematicae}, 223:1--47, 2021.

\bibitem{BL20}
T.~{Budzinski} and B.~{Louf}.
\newblock Local limits of bipartite maps with prescribed face degrees in high
  genus.
\newblock {\em Ann. Probab.}, to appear.

\bibitem{CD06}
P.~Chassaing and B.~Durhuus.
\newblock Local limit of labeled trees and expected volume growth in a random
  quadrangulation.
\newblock {\em Ann. Probab.}, 34(3):879--917, 2006.

\bibitem{ChPerso}
L.~Chen.
\newblock personal communication.

\bibitem{CT20}
L.~Chen and J.~Turunen.
\newblock Critical {I}sing model on random triangulations of the disk:
  Enumeration and local limits.
\newblock {\em Communications in Mathematical Physics}, 374:1577--1643, 2020.

\bibitem{CT20b}
L.~Chen and J.~Turunen.
\newblock Ising model on random triangulations of the disk: phase transition.
\newblock {\em arXiv:2003.09343}, 2020.

\bibitem{CV81}
R.~Cori and B.~Vauquelin.
\newblock Planar maps are well labeled trees.
\newblock {\em Canad. J. Math.}, 33(5):1023--1042, 1981.

\bibitem{CurPSHIT}
N.~Curien.
\newblock Planar stochastic hyperbolic triangulations.
\newblock {\em Probability Theory and Related Fields}, 165(3):509--540, 2016.

\bibitem{C-StFlour}
N.~Curien.
\newblock Peeling random planar maps.
\newblock {\em Saint-Flour lecture notes}, 2019.

\bibitem{CKM21}
N.~Curien, I.~Kortchemski, and C.~Marzouk.
\newblock The mesoscopic geometry of biconditioned planar maps.
\newblock {\em In preparation}.

\bibitem{BD18}
M.~Drmota and B.~Stufler.
\newblock Pattern occurrences in random planar maps.
\newblock {\em Statistics and Probability Letters}, 158:108666, 2020.

\bibitem{GJ08}
I.~P. Goulden and D.~M. Jackson.
\newblock The {KP} hierarchy, branched covers, and triangulations.
\newblock {\em Adv. Math.}, 219(3):932--951, 2008.

\bibitem{HS33}
T.~H. Hildebrandt and I.~J. Schoenberg.
\newblock On linear functional operations and the moment problem for a finite
  interval in one or several dimensions.
\newblock {\em Annals of Mathematics}, 34(2):317--328, 1933.

\bibitem{Kri07}
M.~Krikun.
\newblock Explicit enumeration of triangulations with multiple boundaries.
\newblock {\em Electron. J. Combin.}, 14(1):Research Paper 61, 14 pp.
  (electronic), 2007.

\bibitem{MM07}
J.-F. Marckert and G.~Miermont.
\newblock Invariance principles for random bipartite planar maps.
\newblock {\em Ann. Probab.}, 35(5):1642--1705, 2007.

\bibitem{St18}
R.~Stephenson.
\newblock Local convergence of large critical multi-type {G}alton--{W}atson
  trees and applications to random maps.
\newblock {\em Journal of Theoretical Probability}, 31(1):159--205, Mar 2018.

\bibitem{CT20c}
J.~Turunen.
\newblock Interfaces in the vertex-decorated {I}sing model on random
  triangulations of the disk.
\newblock {\em arXiv:2003.11012}, 2020.

\bibitem{Tut62}
W.~T. Tutte.
\newblock A census of planar triangulations.
\newblock {\em Canad. J. Math.}, 14:21--38, 1962.

\end{thebibliography}

\end{document}